\documentclass{amsart}
\usepackage{color}
\usepackage{graphicx,epsf}
\usepackage{amsmath,amscd}
\newtheorem{theorem}{Theorem}%[section]
\newtheorem{lemma}{Lemma}
\newtheorem{corollary}{Corollary}
\newtheorem{proposition}{Proposition}
\theoremstyle{definition}
\newtheorem{definition}{Definition}
\newtheorem{example}{Example}
\newtheorem{notation}[theorem]{Notation}
\newtheorem{conjecture}{Conjecture}

\theoremstyle{remark}
\newtheorem{remark}{Remark}
\numberwithin{equation}{section}
\newcommand{\Real}{{\mathbb R}}
\newcommand{\Z}{\mathbb Z}

\newcommand{\eps}{\varepsilon}
\newcommand{\f}{\mathbf{f}}
\newcommand{\x}{\mathbf{x}}
\newcommand{\y}{\mathbf{y}}
\newcommand{\z}{\mathbf{z}}

\newcommand {\hide}[1]{}

\begin{document}
\title[Monotone functions and maps]{Monotone functions and maps}

\author{Saugata Basu}
\address{Department of Mathematics,
Purdue University, West Lafayette, IN 47907, USA}
\email{sbasu@math.purdue.edu}
\author{Andrei Gabrielov}
\address{Department of Mathematics,
Purdue University, West Lafayette, IN 47907, USA}
\email{agabriel@math.purdue.edu}
%\thanks{}
\author{Nicolai Vorobjov}
\address{
Department of Computer Science, University of Bath, Bath
BA2 7AY, England, UK}
\email{nnv@cs.bath.ac.uk}
\thanks{2000 Mathematics Subject Classification 14P15, 57Q99.}

\keywords{o-minimal structure, semi-monotone set, monotone map, regular cell}

\dedicatory{To Professor Heisuke Hironaka on the occasion
     of his $80$th birthday}
% ----------------------------------------------------------------

\begin{abstract}
In \cite{BGV} we defined semi-monotone sets, as open bounded sets,
definable in an o-minimal structure over the reals
(e.g., real semialgebraic or subanalytic sets), and
having connected intersections with all
translated coordinate cones in $\Real^n$.
In this paper we develop this theory further by defining
{\em monotone functions} and {\em maps}, and studying their fundamental geometric properties.
We prove several equivalent conditions for a bounded continuous definable function or map
to be monotone.
We show that the class of graphs of monotone maps is closed under intersections with
affine coordinate subspaces and projections to coordinate subspaces.
We prove that the graph of a monotone map is a topologically regular cell.
These results generalize and expand the corresponding results obtained in \cite{BGV}
for semi-monotone sets.
\end{abstract}
\maketitle

% ----------------------------------------------------------------
\section*{Introduction}

This paper is a continuation of the work initiated in an earlier paper
\cite{BGV} where the authors introduced a particular class of sets, called semi-monotone,
being open, bounded and definable in an o-minimal structure over the reals
(e.g., real semialgebraic or subanalytic).
One of the main results in \cite{BGV} is that semi-monotone sets are topologically regular cells.
Here we generalize this result to the sets of any codimension.
The immediate motivation for defining this class of definable sets
was to prove the existence of definable triangulations ``compatible''  with
a given definable function -- more precisely, the following conjecture.

\begin{conjecture}[\cite{BGV}]
\label{con:triang}
Let $f:\> K \to \Real$, be a definable function on a compact definable set $K \subset \Real^m$.
Then there exists a definable triangulation of $K$ such that, for each $n\le\dim K$ and for each
open $n$-simplex $\Delta$ of the triangulation,
\begin{enumerate}
\item
the graph $\Gamma :=\{(\x,t)|\> \x\in\Delta,\,t=f(\x)\}$ of the restriction of
$f$ on $\Delta$ is a topologically regular $n$-cell (see Definition~\ref{def:regular});
\item
either $f$ is a constant on $\Delta$ or
each non-empty level set $\Gamma \cap\{t= {\rm const} \}$ is a topologically regular $(n-1)$-cell.
\end{enumerate}
\end{conjecture}

Conjecture \ref{con:triang} is part of a larger program of obtaining
combinatorial classification of monotone families of definable sets discussed in \cite{BGV}.
The triangulation described in the conjecture can be viewed as a
topological resolution of singularities of definable functions.

The role of semi-monotone sets in the proposed proof of the above conjecture is as follows.
We first hope to prove the existence of a definable, regular cell decomposition of $K$,
such that the properties (i) and (ii) are satisfied for each cell of the decomposition.
The triangulation will then be obtained by generalized barycentric
subdivision of these cells. In order for such an approach to work, one needs
a good supply of definable cells guaranteed to be regular.

The semi-monotone sets fit this requirement.
Non-empty semi-monotone sets are topologically regular cells \cite{BGV}.
Moreover, the class of semi-monotone sets is stable under maps that permute the coordinates
of $\Real^n$.
However, non-empty
semi-monotone sets are open (and hence full dimensional) definable subsets of $\Real^n$.
In this paper, we introduce a certain class of definable maps $\f:X \to \Real^k$ where $X$ is
a semi-monotone subset of $\Real^n$.
We call these maps {\em monotone maps} (see Definition \ref{def:monot_map} below).
We give several characterizations of monotone maps.
Our main result (Theorem~\ref{th:regularcell} below) states that the
graphs of monotone maps are topologically regular cells.
This implies the topological regularity of toric cubes (Corollary~\ref{cor:toric})
and settles a conjecture proposed in \cite{Sturmfelsetal2012}.

We also prove that monotone maps satisfy a suitable generalization
of the coordinate exchange property satisfied by semi-monotone sets -- namely,
if ${\bf F} \subset \Real^n \times \Real^k$ is a graph of a monotone map
$\f:\> X \to \Real^k$, then for any subset of $n$ coordinates such
that the image $X'$ of ${\bf F}$ under projection to the span of these coordinates
is $n$-dimensional, $X'$ is a semi-monotone set, and
${\bf F}$ is the graph of a monotone map on $X'$ (see Theorem \ref{th:exchange} below).

For $k=0$, we recover the main statements about semi-monotone sets proved  in \cite{BGV}.
Moreover, the proof here is simpler than in \cite{BGV}.
As a result we now have a full supply of regular cells (of all dimensions),
and hence we are a step nearer to the proof of Conjecture~\ref{con:triang}.

Note that Conjecture~\ref{con:triang} does not follow from results in the literature
on the existence of definable triangulations adapted to a given finite
family of definable subsets of $\Real^n$ (such as \cite{VDD, Coste}), since all the proofs
use a preparatory linear change of coordinates in order for the given definable sets
to be in a good position with respect to coordinate projections.
Since we are concerned with the graphs and the level sets of a function, in order
to prove Conjecture~\ref{con:triang} we are not allowed to make any change of coordinates
which involves the last coordinate.
Paw{\l}ucki \cite{Pawlucki} has considered the problem of obtaining a regular
cell decomposition with a restriction on the allowed change in coordinates --
namely, only permutations of the coordinates are allowed. In this setting Paw{\l}ucki
obtains a decomposition whose full dimensional cells are regular.
Note that even if this decomposition can be carried through so that all cells (including
those of positive codimension) are regular, it would not be enough for our
purposes since we cannot allow a change of the last coordinate.

\subsection*{Acknowledgements}
The first author was supported in part by NSF grant CCF-0915954.
The second author was supported in part by NSF grants DMS-0801050 and DMS-1067886.

\section{Semi-monotone sets}

In what follows we fix an o-minimal structure over $\Real$, and consider only sets and
maps that are definable in this structure (unless explicitly stated otherwise).
The reader may assume that all sets and maps in this paper are either real
semialgebraic or subanalytic.

\begin{definition}\label{def:semi-monotone}
Let $L_{j, \sigma, c}:= \{ \x=(x_1, \ldots ,x_n) \in \Real^n|\> x_j \sigma c \}$
for $j=1, \ldots ,n$, $\sigma \in \{ <,=,> \}$, and $c \in \Real$.
Each intersection of the kind
$$C:=L_{j_1, \sigma_1, c_1} \cap \cdots \cap L_{j_m, \sigma_m, c_m} \subset \Real^n,$$
where $m=0, \ldots ,n$, $1 \le j_1 < \cdots < j_m \le n$, $\sigma_1, \ldots ,\sigma_m \in \{<,=,> \}$,
and $c_1, \ldots ,c_m \in \Real$, is called a {\em coordinate cone} in $\Real^n$.

Each intersection of the kind
$$S:=L_{j_1, =, c_1} \cap \cdots \cap L_{j_m, =, c_m} \subset \Real^n,$$
where $m=0, \ldots ,n$, $1 \le j_1 < \cdots < j_m \le n$,
and $c_1, \ldots ,c_m \in \Real$, is called an {\em affine coordinate subspace} in $\Real^n$.

In particular, the space $\Real^n$ itself is both a coordinate cone and an affine coordinate
subspace in $\Real^n$.
\end{definition}

Throughout the paper we assume that the empty set is connected.

\begin{definition}[\cite{BGV}]\label{def:set}
An open (possibly, empty) bounded set $X \subset \Real^n$ is called {\em semi-monotone} if
for each coordinate cone $C$  the intersection $X \cap C$ is connected.
\end{definition}

\begin{proposition}[\cite{BGV}, Lemma~1.2, Corollary~1.4]\label{prop:proj}
The projection of a semi-monotone set $X$ on any coordinate subspace, and the
intersection $X \cap C$ with a coordinate cone $C$ are semi-monotone sets.
\end{proposition}

The following necessary and sufficient condition of semi-monotonicity shows that in the
definition it is enough to consider the intersections of $X$ with affine coordinate subspaces.

\begin{theorem}\label{th:def2}
An open (possibly, empty) bounded set $X \subset \Real^n$ is semi-monotone if and only if
for each affine coordinate subspace $S$ the intersection $X \cap S$ is connected.
\end{theorem}

\begin{lemma}\label{le:connected}
If $X \subset \Real^n$ is any connected definable set such that for some $j \in \{ 1, \ldots n \}$
and each $b \in \Real$ the intersection $X \cap \{x_j=b\}$ is connected,
then the sets $X \cap \{x_j<c\}$ and $X \cap \{x_j>c\}$ are connected for all
$c \in \Real$.
\end{lemma}

\begin{proof}
Observe that connectedness is equivalent to path-connectedness for definable sets.
Consider any two points $\y,\> \z \in X \cap \{x_j<c\}$, then there is a path $\gamma \subset X$
connecting them.
Suppose, for definiteness, that $y_j \le z_j$.
Let ${\bf w}$ be the point in $\gamma \cap X \cap \{ x_j = z_j \}$ which is closest to $\y$ in
$\gamma$.
Then the union of the segment of $\gamma$ between $\y$ and ${\bf w}$, and
a path in $X \cap \{ x_j = z_j \}$, that connects ${\bf w}$ with
$\z$, is a path in $ X \cap \{x_j<c\}$ connecting $\y$ with $\z$.

The similar argument shows that $X \cap \{x_j>c\}$ is path-connected.
\end{proof}

\begin{proof}[Proof of Theorem\ref{th:def2}]
If $X$ is semi-monotone, then $X \cap S$ is always connected by the definition.

To prove the converse, observe that since $X$ is connected, and $X \cap \{x_j=b\}$ is connected
for every $j=1, \ldots ,n$ and every $b \in \Real$, the intersections $X \cap \{ x_j<c \}$ and
$X \cap \{ x_j>c \}$ are connected for every $c \in \Real$,
by Lemma~\ref{le:connected}.
The theorem follows, by the induction on the number of half-spaces which form a coordinate cone,
since these intersections can then be taken as $X$.
\end{proof}

\begin{corollary}\label{cor:def2}
An open (possibly, empty) bounded connected set $X \subset \Real^n$ is semi-monotone if and only if
the intersection $X \cap L_{j,=,c}$ is semi-monotone for every $j=1, \ldots ,n$ and every
$c \in \Real$.
\end{corollary}

\begin{proof}
The statement easily follows from Theorem~\ref{th:def2} by induction on $n$.
\end{proof}

\begin{definition}\label{def:sub-super}
A bounded upper semi-continuous function $f$ defined on a non-empty semi-monotone set
$X \subset \Real^n$ is {\em submonotone} if, for any $b \in \Real$, the set
$\{ \x \in X|\> f(\x)<b \}$ is semi-monotone.
A function $f$ is {\em supermonotone} if $(-f)$ is submonotone.
\end{definition}

\begin{notation}
Let the space $\Real^n$ have coordinate functions $x_1, \ldots ,x_n$.
Given a subset $I= \{ x_{j_1}, \ldots ,x_{j_m} \} \subset \{ x_1, \ldots , x_n \}$,
let $W$ be the linear subspace of $\Real^n$ where all coordinates in $I$ are equal to zero.
By a slight abuse of notation we will denote by ${\rm span} \{ x_{j_1},\ldots,x_{j_m} \}$
the quotient space $\Real^n/W$.
Similarly, for any affine coordinate subspace $S \subset \Real^n$
on which all the functions $x_j \not\in I$ are constant, we will identify $S$ with its
image under the canonical surjection to $\Real^n/W$.
\end{notation}

\begin{lemma}\label{le:infsup}
Let the function $f: X \to \Real$ be submonotone (respectively, supermonotone), and let
$X'$ be the image of the projection of $X$ to ${\rm span} \{ x_1, \ldots ,x_{n-1} \}$.
Then the function $\inf_{x_n}f: X' \to \Real$ (respectively, $\sup_{x_n}f: X' \to \Real$) is
submonotone (respectively, supermonotone).
\end{lemma}

\begin{proof}
According to Proposition~\ref{prop:proj}, the set $X'$ is semi-monotone.
Assume that $f$ is submonotone.
Then for any $b \in \Real$ the image $X_b'$ of the projection of $\{ \x \in X|\> f(\x)<b \}$ to
${\rm span} \{ x_1, \ldots ,x_{n-1} \}$ coincides with
$\{ (x_1, \ldots ,x_{n-1}) \in X'|\> \inf_{x_n}f (\x) <b \}$.
Since, by Proposition~\ref{prop:proj}, $X_b'$ is semi-monotone, the function $\inf_{x_n} f$
satisfies the definition of submonotonicity.

The proof that $\sup_{x_n} f$ is supermonotone is analogous.
\end{proof}

\begin{proposition}[\cite{BGV}, Theorem~1.7]\label{prop:def_semi}
An open non-empty bounded set $X \subset \Real^n$ is semi-monotone if and only if it
satisfies the following conditions.
If $X \subset \Real^1$ then $X$ is an open interval.
If $X \subset \Real^n$ then
$$X= \{ (\x,y)|\> \x \in X',\> f(\x) <y< g(\x) \}$$
for a submonotone function $f$ and a supermonotone function $g$, both defined on a semi-monotone set
$X' \subset \Real^{n-1}$, with $f(\x)< g(\x)$ for all $\x \in X'$.
\end{proposition}

%It was proved in \cite{BGV} that every semi-monotone set is a regular cell.
%In Section~\ref{sec:regular} we generalize this theorem to graphs of {\em monotone maps}
%(defined below), and give, even for the semi-monotone sets, a more direct and elementary proof.

%%%%%%%%%%%%%%%%%%%%%%%%%%%%%%%%%%%%%%%%%%%%%%%%%%%%%%%%%%%%%%%%%%%%%%%%%%%%%
%%sb adds
%%%%%%%%%%%%%%%%%%%%%%%%%%%%%%%%%%%%%%%%%%%%%%%%%%%%%%%%%%%%%%%%%%%%%%%%%%%%%
The rest of the paper is organized as follows.

In Section~\ref{sec:monotone_functions}, we define the class of monotone functions.
These are a special type of definable functions $f:\> X \to \Real$ where $X$ is any non-empty
semi-monotone set.
We give several different characterizations of monotone functions (Lemma~\ref{le:def_monotone},
Corollary~\ref{cor:def_monotone}, and Theorem~\ref{th:monot_funct}).
In particular, Lemma \ref{le:def_monotone} should be compared with the
Definition~\ref{def:set} above of semi-monotone sets, and Theorem~\ref{th:monot_funct}
should be compared with the corresponding result, Corollary~\ref{cor:def2}, for semi-monotone sets.
We also prove a few useful topological results in this section.
In particular, we prove
%%sb changes
a topological property of semi-monotone sets
and graphs of monotone functions that could be
viewed as
an analog of Sch\"onflies Theorem for semi-monotone
sets (see Lemma~\ref{le:divide} below).

In Section~\ref{sec:monotone_maps}, we generalize the definition of monotone functions and
define monotone maps $\f:\> X \to \Real^k$, where $X \subset \Real^n$ is a non-empty semi-monotone
subset of $\Real^n$ (see Definition~\ref{def:monot_map} below).
The definition is inductive (induction on $n$) and is more complicated than the definitions of
semi-monotone sets and monotone functions.
The combinatorial information regarding the dependence or independence
of the map $\f$  with respect to the various coordinates is more subtle and is recorded in a
matroid, $\bf m$,  of rank $n$ (see Theorem \ref{th:matroid}), which is associated with $\f$.
We prove several important properties of monotone maps and their associated matroids in
Section~\ref{sec:monotone_maps}.
In particular, we show that if ${\bf F} \subset \Real^{n+k}$ is the graph of a monotone map
$\f:\> X \to \Real^k$, where $X \subset \Real^n$, and
$I \subset \{ x_1, \ldots ,x_n,y_1, \ldots ,y_k \}$ is a basis of its associated matroid,
then the image of ${\bf F}$ under the projection to  ${\rm span}\> I$ is a semi-monotone set, and
${\bf F}$ is also the graph of a monotone map defined on this set.
We also identify a key property of monotone maps, of being \emph{quasi-affine}
(Definition~\ref{def:quasi-affine}, and Theorem~\ref{th:vdd}) which will be used later in an
essential way.

In Section~\ref{sec:equivalent_defs_of_monotone_maps}, we prove several different characterizations
of monotone maps including Theorem \ref{th:def_monotone_map} (which generalizes
Lemma~\ref{le:def_monotone} from functions to maps) and Theorem~\ref{th:monot_maps}
(generalizing similarly Theorem~\ref{th:monot_funct} from functions to maps).
We also prove
a topological result
%%sb
%%an analog of Sch\"onflies Theorem for graphs of monotone maps,
namely
Theorem \ref{th:schonflies_for_monotone} (generalizing Lemma \ref{le:divide}).

It was proved in \cite{BGV} that every semi-monotone set is a topologically regular cell.
In Section~\ref{sec:regular} we generalize this theorem to graphs of monotone maps
(see Theorem~\ref{th:regularcell}).
The proof of Theorem~\ref{th:regularcell} is new even in the case of semi-monotone sets,
and simpler, as it avoids a more advanced machinery from PL topology that was used in \cite{BGV}.

In Section~\ref{sec:toric} we give an application of Theorem~\ref{th:regularcell}, proving that
{\em toric cubes}, introduced in \cite{Sturmfelsetal2012}, are topologically regular cells.

Section~\ref{sec:appendix} contains a digest of some propositions, mostly from PL topology,
which are used in the proofs in the preceding sections.

\section{Monotone functions}\label{sec:monotone_functions}

\begin{definition}[\cite{VDD}]\label{def:strictly}
A definable function $f$ on a non-empty open set $X \subset \Real^n$ is called {\em strictly
increasing in the coordinate} $x_j$, where $j=1, \ldots, n$, if for any two points $\x, \y \in X$
that differ only in the coordinate $x_j$, with $x_j <y_j$, we have $f(\x) < f(\y)$.
Similarly we define the notions of $f$ {\em strictly decreasing in the coordinate} $x_j$ and
$f$ {\em independent of the coordinate} $x_j$, the latter meaning that $f(\x)=f(\y)$ whenever
$\x, \y \in X$ differ only in the coordinate $x_j$.
\end{definition}

\begin{definition}\label{def:monotone}
A definable function $f$ defined on a non-empty semi-monotone set $X \subset \Real^n$ is called
{\em monotone} if it is
\begin{itemize}
\item[(i)]
both sub- and supermonotone (in particular, bounded and continuous, see
Definition~\ref{def:sub-super});
\item[(ii)]
either strictly increasing in, or strictly decreasing
in, or independent of $x_j$, for each $j=1, \ldots ,n$.
\end{itemize}
\end{definition}

\begin{example}
The function $x_1^2 + x_2^2$  on the semi-monotone set
$$X=\{ x_1>0,\> x_2 >0,\> x_1+x_2<1 \} \subset \Real^2$$
satisfies (ii) in Definition~\ref{def:monotone}, is submonotone but not supermonotone.
Hence this function is not monotone.
On the other hand, the function $x_1^2 + x_2^2$ on the semi-monotone set $(0,1)^2$ is monotone.
\end{example}

\begin{remark}\label{re:monotone_funct}
It follows from the definition that the restriction of a monotone function $f$ to a non-empty
set $X \cap \{ x_j=c \}$ for any $j=1, \ldots, n$ and $c \in \Real$ is a monotone function
in $n-1$ variables.
Lemma~\ref{le:def_monotone} below implies that the restriction of $f$ to a non-empty
$X \cap C$, where $C$ is a coordinate cone in $\Real^n$, is also a monotone function.
However, as exhibited in Example 2.3, the restriction of a monotone function
$f:\> X \to \Real$ to a semi-monotone subset $Y \subset X$ is not necessarily monotone.
\end{remark}

\begin{example}
The function on the semi-monotone set $X= (0,1) \times (-1,1) \subset \Real^2$
defined as:
$$x_1x_2\> \text{when}\> x_2 \ge 0,\> \text{and}\>
(1-x_1)x_2\> \text{when}\> x_2 \le 0,$$
is sub- and supermonotone, strictly increasing in $x_2$ on $X$, strictly increasing in $x_1$ on
$X \cap \{ x_2 \neq 0 \}$, but is constant on $X \cap \{ x_2=0 \}$.
Hence this function is not monotone.
\end{example}

\begin{definition}\label{def:nonconst}
We say that a monotone function $f$ is {\em non-constant} in $x_j$ if it is either strictly
increasing or strictly decreasing in $x_j$.
\end{definition}

Let a monotone function $f:\> X \to \Real$ on a semi-monotone set $X \subset \Real^n$
be non-constant in $x_n$.
Let
$$F:= \{(\x, y)|\> \x \in X,\, y=f(\x) \} \subset \Real^{n+1}$$
be the graph of $f$ and
$U$ be the projection of $F$ to ${\rm span} \{x_1, \ldots ,x_{n-1},y \}$.

\begin{lemma}\label{le:Usemi-monotone}
The set $U$ is semi-monotone, and
$$F= \{ (\x, y)|\> \x \in X,\, x_n=g(x_1, \ldots ,x_{n-1},y) \}$$
is the graph of a continuous function $g$ on $U$.
\end{lemma}

\begin{proof}
Since the projection of $U$ to ${\rm span} \{ x_1, \ldots ,x_{n-1} \}$,
coincides with the projection $X'$ of $X$ to the same space, it is a semi-monotone set by
Proposition~\ref{prop:proj}.

Observe that
$$U=\{ (x_1, \ldots ,x_{n-1},y)|\>  (x_1, \ldots ,x_{n-1}) \in X',\>
\inf_{x_n}f < y < \sup_{x_n}f \}.$$

According to Lemma~\ref{le:infsup}, $\inf_{x_n}f$ is submonotone and $\sup_{x_n}f$ is supermonotone.
Moreover, $\sup_{x_n}f(x_1, \ldots ,x_{n-1}) < \inf_{x_n}f(x_1, \ldots ,x_{n-1})$,
for each $(x_1, \ldots ,x_{n-1}) \in X'$, since $f$ is non-constant in $x_n$.
Therefore  the set $U$ is semi-monotone, by Proposition~\ref{prop:def_semi}.

The function $g$ is defined, since $f$ is non-constant in $x_n$, and
continuous since $f$ is continuous and these functions have the same graph $F$.
\end{proof}

\begin{lemma}\label{le:def_monotone}
Let $f$ be a bounded continuous function defined on an open bounded non-empty set $X \subset \Real^n$,
either strictly increasing in, strictly decreasing
in, or independent of $x_j$, for each $j=1, \ldots ,n$.
Let $F$ be the graph of $f$.
The following three statements are equivalent.
\begin{itemize}
\item[(i)]
The function $f$ is monotone.
\item[(ii)]
For each coordinate cone $C$ in $\Real^{n+1}$ the intersection $C \cap F$ is connected.
\item[(iii)]
For each affine coordinate subspace $S$ in $\Real^{n+1}$ the intersection $S \cap F$ is connected.
\end{itemize}
\end{lemma}

\begin{proof}
We first prove that (i) is equivalent to (ii).

Let $f$ be monotone (in particular, $X$ is semi-monotone), and let $C$ be a coordinate cone in
$\Real^{n+1}$.
It is sufficient to consider the cases when $C$ is defined by a sign condition on the variable $y$,
otherwise, by Proposition~\ref{prop:proj}, the situation is reduced to $f$ defined on a
smaller semi-monotone set,
the intersection of $X$ with a coordinate cone in ${\rm span} \{ x_1, \ldots ,x_n \}$.
If $C= \{ y <c \}$ for some $c \in \Real$, then, since $f$ is submonotone, the
projection of $C \cap F$ to ${\rm span} \{ x_1, \ldots ,x_n \}$ is semi-monotone,
hence connected.
Since $f$ is continuous, the pre-image of this projection in $F$ is connected.
Similar argument applies in the case when $C= \{ y >c \}$.

Suppose that $C= \{ y =c \}$.
Due to Lemma~\ref{le:Usemi-monotone}, the intersection $U \cap \{ y=c \}$ is semi-monotone,
hence connected, and since the function $g$ is continuous, the pre-image of $U \cap \{ y=c \}$
in $F$ is connected.

Conversely, let for each coordinate cone $C$ in $\Real^{n+1}$ the intersection
$C \cap F$ be connected.
Let $C'$ be a coordinate cone in ${\rm span} \{ x_1, \ldots ,x_n \}$.
Then the intersection $F \cap (C' \times \Real)$ is connected, hence the image of its projection,
$C' \cap X$, is connected.
It follows that $X$ is semi-monotone.
We need to prove that $f$ is both sub- and supermonotone.
Let $c \in \Real$.
Then the set $C' \cap \{ f <c \}$ is the image under the projection to $\Real^n$ of
the connected set $C \cap F$ where $C:=(C' \times \Real) \cap \{ y < c \}$.
Since $f$ is continuous, $C' \cap \{ f <c \}$ is connected, hence $f$ is submonotone.
The similar arguments show that $f$ is supermonotone.

Now we prove that (ii) is equivalent to (iii).

If (ii) is satisfied, then for each $S$ the intersection $S \cap F$ is connected, since
$S$ ia a particular case of the coordinate cone.
The converse follows from Lemma~\ref{le:connected} by a straightforward induction on the
number of strict inequalities defining the coordinate cone.
\end{proof}

\begin{corollary}\label{cor:def_monotone}
Under the conditions of Lemma~\ref{le:def_monotone} and assuming $X$ connected,
the non-constant function $f$ is monotone if and only if
\begin{itemize}
\item[(i)]
for every $x_j$ and every $c \in \Real$ the intersection $F \cap \{ x_j=c \}$
is either empty, or the graph of a monotone function in $n-1$ variables, and
\item[(ii)]
for every $b \in \Real$ the intersection $F \cap \{ y=b \}$ is either empty,
or the graph of a monotone function in $n-1$ variables.
\end{itemize}
\end{corollary}

\begin{proof}
The statement easily follows from Lemma~\ref{le:def_monotone} by induction on $n$.
\end{proof}

\begin{remark}
In Lemma~\ref{le:function_graph} we will prove that the requirement (ii) alone in
Corollary~\ref{cor:def_monotone} is a necessary and sufficient for $f$ to be monotone.
In Theorem~\ref{th:monot_funct} we will show that {\em fixing} any $j$ in the part (i) of
Corollary~\ref{cor:def_monotone}, and adding the requirement for $\inf_{x_j} f$ and $\sup_{x_j} f$
to be sub- and supermonotone functions respectively, makes (i) also a necessary and sufficient
condition for $f$ to be monotone.
\end{remark}

\begin{corollary}
Let $f:\> X \to \Real$ be a monotone function having the graph $F \subset \Real^{n+1}$.
Then for every $z \in \{ x_1, \ldots ,x_n, y \}$
and every $c \in \Real$ each of the intersections $F \cap \{ z < c \}$ and
$F \cap \{ z > c \}$ is either empty or the graph of a monotone function.
\end{corollary}

\begin{proof}
%Consider, for example the case of a non-empty set $F \cap \{ y < b \}$ (all other cases are similar).
According to part (ii) of Lemma~\ref{le:def_monotone}, for any affine coordinate subspace
$S \subset \Real^{n+1}$, the intersection $F \cap \{ z < c \} \cap S$ is connected, since
$\{ z < c \} \cap S$ is a coordinate cone.
Then part (iii) of this lemma implies that $F \cap \{ z < c \}$ is either empty or the graph
of a monotone map.
The case of $F \cap \{ z > c \}$ completely analogous.
\end{proof}

\begin{lemma}\label{le:exchange}
In the conditions of Lemma~\ref{le:Usemi-monotone}, the function $g(x_1, \ldots ,x_{n-1},y)$
is monotone.
\end{lemma}

\begin{proof}
The function $g$ is defined on the semi-monotone set $U$.
It has the same graph as the monotone function $f$, and hence, by Lemma~\ref{le:def_monotone},
is itself monotone.
\end{proof}

\begin{remark}\label{re:g}
The function $g$ was constructed from $f$ with respect to the variable $x_n$.
An analogous function $g_j(x_1, \ldots, x_{j-1},y,x_{j+1}, \dots, x_n)$ can be constructed from
$f$ with respect to any variable $x_j$ in which $f$ is non-constant, and
Lemma~\ref{le:exchange} implies that $g_j$ is monotone.
If $g_j$ is non-constant in a variable $x_\ell$,
then the function constructed from $g_j$ with respect to $x_\ell$, coincides with
the function
$$g_\ell(x_1, \ldots ,x_{\ell -1},y,x_{\ell+1}, \ldots ,x_n).$$
The function constructed from $g_j$ with respect to $y$ coincides with $f$.
All these functions have the same graph $F$ as $f$.
\end{remark}

\begin{lemma}\label{le:components}
Let $X$ be an open, simply connected subset in $\Real^n$, and let
$\Sigma \subset X$ be a non-empty connected $(n-1)$-dimensional manifold closed in $X$.
Then $X \setminus \Sigma$ has two connected components.
\end{lemma}

\begin{proof}
There is a short exact sequence (a combination of
a cohomological exact sequence of the pair $(X, \Sigma)$ and the Poincare duality)
$$0=H_1(X)\to H_0(\Sigma)\to H_0(X\setminus\Sigma)\to H_0(X)\to 0$$
which, given $H_0(\Sigma)=H_0(X)= \Z$, implies that ${\rm rank} (H_0(X\setminus\Sigma))=2$.
\end{proof}

\begin{remark}
Here is an alternative proof of Lemma~\ref{le:components}, not using an exact sequence.

If $\Sigma$ is not orientable, choose a normal at a point $x \in \Sigma$
and find a path in $\Sigma$ that changes its orientation.
Lift a path in the direction of the normal and connect its ends.
The result is a loop in $X$ intersecting $\Sigma$ transversally at $x$.
This loop cannot be contractible in $X$ since its intersection index with $\Sigma$ is $\pm 1$.
It follows that $\Sigma$ is orientable.

If $X \setminus \Sigma$ is connected, take a segment transversal to $\Sigma$ and connect
its ends in $X\setminus\Sigma$.
We get a loop in $X$ which intersection index with $\Sigma$ is $\pm 1$.
Thus, $X \setminus \Sigma$ is not connected.

Assume $\Sigma$ is oriented.
Every point $x \in X \setminus \Sigma$ can be connected in $X \setminus \Sigma$ to a point
$v\in \Sigma$ by a path $\gamma$ such that $\gamma \setminus \{v\} \subset X \setminus \Sigma$.
If the path $\gamma'$ for a point $x'$ gets to $\Sigma$ at the point $v'$ from the same side
of $\Sigma$ as $x$, connect $v$ and $v'$ by a path $\rho$ in $\Sigma$, then
lift $\rho$ along the normals to $\Sigma$.
We get a path connecting $x$ and $x'$ in $X \setminus \Sigma$.
It follows that $X \setminus \Sigma$ has exactly two connected components.
\end{remark}

\begin{lemma}\label{le:divide}
Let $X$ be a semi-monotone set in $\Real^n$ and $\Sigma \subset X$ a graph of a monotone
function $x_n=h_n(x_1,\dots,x_{n-1})$ on some semi-monotone $Y \subset \Real^{n-1}$,
such that $\partial \Sigma \subset \partial X$.
Then $X \setminus \Sigma$ is a union of two semi-monotone sets.
\end{lemma}

\begin{proof}
First notice that, by Lemma~\ref{le:components}, $X \setminus \Sigma$ has two connected
components, $X_+$ and $X_-$.
For any variable $x_j,\> j=1, \ldots ,n$, and any $c \in \Real$ the intersection
$X \cap \{ x_j =c \}$ is semi-monotone due to Corollary~\ref{cor:def2}, while $\Sigma \cap \{ x_j =c \}$
is either empty or the graph of a monotone function due to Corollary~\ref{cor:def_monotone}.

The rest of the proof is by induction on $n$, the base for $n=1$ being trivial.
If $h_n$ is constant in each variable $x_1, \ldots ,x_{n-1}$ then the statement of the theorem
is trivially true.
Let $X \cap  \{ x_j=c \}$ be non-empty for some variable $x_j,\> j=1, \ldots ,n$, and some $c \in \Real$.
Note that if $h_n$ is non-constant in $x_j$, where $j<n$, then according to Remark~\ref{re:g},
$\Sigma$ is the graph of a monotone function
$x_j=h_j(x_1, \ldots ,x_{j-1},x_{j+1}, \ldots ,x_n)$ on some semi-monotone set $Y_j$.
Now let $j=1, \ldots, n$.
If $\Sigma \cap \{ x_j=c \}= \emptyset$, then either $X_+ \cap  \{ x_j=c \}= X \cap  \{ x_j=c \}$
or $X_- \cap  \{ x_j=c \}= X \cap  \{ x_j=c \}$, in any case the intersection is semi-monotone.
Assume now that $\Sigma \cap \{ x_j=c \} \neq \emptyset$.
Observe that $\Sigma \not\subset \{ x_j=c \}$, since $h_n$ is not a constant function,
and $\Sigma$ is the graph of $h_n$.
Hence, $(X_+ \cup X_-) \cap \{ x_j =c \} \neq \emptyset$.

Both intersections, $X_+ \cap \{ x_j=c \}$ and $X_- \cap \{ x_j=c \}$ are non-empty.
Indeed, if $j<n$ and $\{x_j=c\}$ contains a point $p \in \Sigma$, it also includes an open interval
of a straight line parallel to $x_n$-axis containing $p$.
The two parts into which $p$ divides that interval belong one to $X_+$ and another to $X_-$.
If $j=n$ then the restriction of the function $h_n$ to a straight line parallel to $x_i$-axis, such
that $h_n$ is non-constant in $x_i$, has the graph which is a subset of $\Sigma$ and has
non-empty intersections with both $\{ x_n <0 \}$ and $\{ x_n >0 \}$.

By the inductive hypothesis, $(X \cap \{ x_j =c \}) \setminus (\Sigma \cap \{ x_j =c \})$ is a union
of two semi-monotone sets.
It follows that one of the connected components of
$(X \cap \{ x_j =c \}) \setminus (\Sigma \cap \{ x_j =c \})$
lies in $X_+$ while another in $X_-$.
Hence, the intersection of $\{ x_j=c \}$ with each of $X_+$ and $X_-$ is semi-monotone.

Finally, in the case when $h_n$ is independent of $x_j,\> j<n$,
the set $\Sigma$ is a cylinder over the graph $\Sigma \cap \{ x_j=c \}$ of a monotone function,
for any $c \in \Real$, and therefore one of the connected components of
$(X \cap \{ x_j =c \}) \setminus (\Sigma \cap \{ x_j =c \})$ lies in $X_+$ while another $X_-$.
It follows that the intersection of $\{ x_j=c \}$ with each of $X_+$ and $X_-$ is semi-monotone.

Corollary~\ref{cor:def2} now implies that each of $X_+$ and $X_-$ is semi-monotone.
\end{proof}

\begin{lemma}\label{le:function_graph}
Let $f:\> X \to \Real$ be a continuous, bounded, non-constant function defined on a non-empty
semi-monotone set $X$.
The function $f$ is monotone if and only if
\begin{itemize}
\item[(i)]
it is either strictly increasing in, or strictly decreasing
in, or independent of $x_j$, for each $j=1, \ldots ,n$;
\item[(ii)]
for every $b \in \Real$ the set $\{ \x \in X|\> f(\x)=b \}$ is either empty, or a graph of a
monotone function in $n-1$ variables.
\end{itemize}
\end{lemma}

\begin{proof}
If $f$ is monotone, then (i) follows from the definition of a monotone function, while (ii)
is the statement of Corollary~\ref{cor:def_monotone}.

Conversely, suppose the conditions (i), (ii) are true.

Let $\Sigma:= X \cap \{ f(\x)=b \}$.
Since $\Sigma$ is a level set of a continuous function, $\partial \Sigma \subset \partial X$.
By Lemma~\ref{le:divide}, $X \setminus \Sigma$ is a union of two semi-monotone sets.
The condition (i) implies that one of these sets is $X \cap \{ f(\x)<b \}$ and another
is $X \cap \{ f(\x)>b \}$.
It follows that $f$ is both sub- and supermonotone.
\end{proof}

\begin{theorem}\label{th:monot_funct}
A continuous function $f$, defined on a non-empty open bounded set $X \subset \Real^n$, and
not independent of $x_n$, is monotone if and only if it satisfies
the following properties:
\begin{itemize}
\item[(i)]
$f$ is either strictly increasing in or strictly decreasing in or independent of each of the
variables $x_j$, where $j=1, \ldots ,n$;
\item[(ii)]
$\inf_{x_n} f$ and $\sup_{x_n} f$ are sub- and supermonotone
functions, respectively, in variables $x_1, \ldots ,x_{n-1}$;
\item[(iii)]
the intersection of $X$ with any straight line parallel to the  $x_n$-axis is either empty or an
open interval;
\item[(iv)]
the restriction of $f$ to each non-empty set $X \cap \{x_n= a\}$, where $a \in \Real$,
is a monotone function.
\end{itemize}
\end{theorem}

\begin{proof}
Assume first that $f$ is not independent of $x_n$ and satisfies the properties (i)--(iii).
Let $F$ be the graph of $f$.
Then $F$ can be represented as in Lemma~\ref{le:Usemi-monotone},
$$F= \{ (\x, y)|\> \x \in X,\, x_n=g(x_1, \ldots ,x_{n-1},y) \},$$
with the function $g$ defined on the domain
$$U=\{ (x_1, \ldots ,x_{n-1},y)|\>  (x_1, \ldots ,x_{n-1}) \in X',\>
\inf_{x_n}f < y < \sup_{x_n}f \},$$
where $X'$ is the projection of $X$ to ${\rm span} \{ x_1, \ldots ,x_{n-1} \}$.
Observe that $F$ is also the graph of $g$.
By the property (ii), and by Proposition~\ref{prop:def_semi}, the domain $U$ is semi-monotone.

Now suppose that $f$ satisfies also the property (iv).
Applying Lemma~\ref{le:function_graph} to $g$, we conclude that this function is monotone.
Hence, by Lemma~\ref{le:def_monotone}, $f$ is also monotone.

Conversely, suppose that a function $f:\> X \to \Real$, not independent of $x_n$, is monotone.
Then properties (i) and (iii) follow from the definition of a monotone function.
By Lemma~\ref{le:Usemi-monotone} the set $U$ is semi-monotone, hence, by Proposition~\ref{prop:def_semi},
the property (ii) is satisfied.
Property (iv) follows immediately from Lemma~\ref{le:def_monotone}.
\end{proof}

\section{Monotone maps}\label{sec:monotone_maps}

\begin{definition}\label{def:matroid}
For a non-empty semi-monotone set $X \subset \Real^n$ and $k \ge 1$, let
$${\bf f}=(f_1, \ldots ,f_k):\> X \to \Real^k$$
be a continuous and bounded map.
Let
$$H:= \{ x_{j_1}, \ldots , x_{j_\alpha}, y_{i_1}, \ldots ,y_{i_\beta}\} \subset
\{ x_1, \ldots ,x_n,y_1, \ldots ,y_k \},$$
where $\alpha+ \beta=n$.
The set $H$ is called a {\em basis} if the map
$$(x_{j_1}, \ldots , x_{j_\alpha}, f_{i_1}, \ldots ,f_{i_\beta}):\> X \to \Real^n$$
is injective.
Thus, a {\em system of basis sets} is associated with $\f$.
\end{definition}

\begin{example}\label{ex:linear}
Let ${\bf f}:\>  X \to \Real^k$ be a linear map on a non-empty semi-monotone set
$X \subset {\rm span} \{x_1, \ldots ,x_n\}$,
and $\bf F$ be its graph, which is an open set in $n$-dimensional linear subspace $L$ in
${\rm span} \{ \x_1, \ldots ,x_n,y_1, \ldots ,y_k \}$.
Let ${\bf b}:=\{{\bf b}_1, \ldots ,{\bf b}_n\}$ be a basis of the space $L$.
Then a set $H:= \{ x_{j_1}, \ldots , x_{j_\alpha}, y_{i_1}, \ldots ,y_{i_\beta}\}$, where
$\alpha + \beta=n$, is a basis for the map $\f$ if and only if
the projection of ${\bf b}$ to the space ${\rm span}\> H$ is a basis of ${\rm span}\> H$.
%A subset $H:= \{ x_{j_1}, \ldots , x_{j_\alpha}, y_{i_1}, \ldots ,y_{i_\beta}\}$,
%where $\alpha + \beta=n$, is a basis for the map
%$\f$ if and only if for any basis $\{ {\bf b}_1, \ldots , {\bf b}_n \}$ of the vector space
%${\rm span} \{ x_{j_1}, \ldots , x_{j_\alpha}, y_{i_1}, \ldots ,y_{i_\beta}\}$
%the projections of vectors ${\bf b}_1, \ldots , {\bf b}_n$ onto $\bf F$ span (form a basis of) $\bf F$.
\end{example}

\begin{lemma}\label{le:function_basis}
If $k=1$, then the system of basis sets associated with $\f=(f_1):\> X \to \Real$ consists of
$\{ x_1, \ldots ,x_n \}$, and each set $\{ x_1, \ldots ,x_{j-1},y,x_{j+1}, \ldots ,x_n\}$ such that
the function $f_1$ is either strictly increasing in $x_j$, or strictly decreasing
in $x_j$ (see Definition~\ref{def:strictly}).
\end{lemma}

\begin{proof}
Clearly, $\{ x_1, \ldots ,x_n \}$ is a basis set.

If $f_1$ is either strictly increasing in $x_j$, or strictly decreasing in $x_j$, then the set
$\{ x_1, \ldots ,x_{j-1},y,x_{j+1}, \ldots ,x_n\}$ is obviously a basis.
Conversely, suppose that $\{ x_1, \ldots ,x_{j-1},y,x_{j+1}, \ldots ,x_n\}$ is a basis set, i.e.,
the restriction of $f_1$ to every non-empty interval
$$X \cap \{ x_1=c_1, \ldots ,x_{j-1}=c_{j-1},x_{j+1}=c_{j+1}, \ldots ,x_n=c_n \},$$
where $c=(c_1, \ldots ,c_{j-1},c_{j+1}, \ldots ,c_n) \in \Real^{n-1}$, is either
strictly increasing or strictly decreasing.
Let $A$ (respectively, $B$) be the set of points $c$ for which the restriction is strictly increasing
(respectively, strictly decreasing), and $X'$ the projection
of $X$ to ${\rm span} \{ x_1, \ldots ,x_{j-1},x_{j+1}, \ldots ,x_n \}$.
Thus, $X'= A \cup B$.
Because $f_1$ is continuous, both sets, $A$ and $B$, are open in $X'$.
Since, by Proposition~\ref{prop:proj}, $X'$ is connected, we conclude that either $X'=A$ or $X'=B$.
\end{proof}

\begin{definition}\label{def:monot_map}
For a non-empty semi-monotone set $X \subset \Real^n$ and $k \ge 1$, let
$${\bf f}=(f_1, \ldots ,f_k):\> X \to \Real^k$$ be a continuous and bounded map and let
${\bf F}:=\{(\x, \y)|\> \x \in X,\> \y={\bf f}(\x) \} \subset \Real^{n+k}$ be its graph.
Associate with ${\bf f}$ a system $\bf m$ of basis sets as in Definition~\ref{def:matroid}.
Define a map $\bf f$ to be {\em monotone}, by induction on $n \ge 1$.

If $n=1$, the map ${\bf f }$ is monotone if for every $i$ the function $f_i$ is monotone.

Assume that monotone maps on non-empty semi-monotone subsets of  $\Real^{n-1}$ are defined.

A map ${\bf f}$ is {\em monotone} if for every $i=1, \ldots ,k$, and every $j=1, \ldots ,n$ such
that $f_i$ is not independent of $x_j$, the following holds.
\begin{itemize}
\item[(i)]
For every $b \in \Real$, the intersection ${\bf F} \cap \{y_i=b \}$
(considered as a set in
${\rm span} \{ x_1, \ldots, x_n,y_1, \ldots ,y_{i-1}, y_{i+1}, \ldots ,y_k \}$),
when non-empty, is the graph of a monotone map, denoted by ${\bf f}_{i,j,b}$,
from a semi-monotone subset of
${\rm span} \{ x_1, \ldots ,x_{j-1},x_{j+1}, \ldots ,x_n \}$, into
${\rm span} \{ y_1, \ldots ,y_{i-1},x_j,y_{i+1}, \ldots ,y_k \}$.
\item[(ii)]
The system of basis sets associated with ${\bf f}_{i,j,b}$ does not depend on $b \in \Real$.
\end{itemize}
\end{definition}

\begin{remark}\label{rem:indep_j}
It follows from Theorem~\ref{th:matroid} that if the conditions in Definition~\ref{def:monot_map}
hold for any one $j$, then they hold for every $j$ such that $f_i$ is not independent of $x_j$.
\end{remark}

\begin{example}
It is easy to check that an affine map ${\bf f}:\>  X \to \Real^k$ on a non-empty semi-monotone set
$X \subset {\rm span} \{x_1, \ldots ,x_n\}$ is a monotone map (cf. Example~\ref{ex:linear}).
\end{example}

\begin{lemma}\label{le:monot-1}
If the map $\f:=(f_1, \ldots ,f_k):\> X \to \Real^k$ is monotone, then
the map $(f_1, \ldots ,f_{k-1}):\> X \to \Real^{k-1}$
is also monotone.
\end{lemma}

\begin{proof}
For any map ${\bf g}:\> X \to \Real^k$, let $[{\bf g}]:\> X \to \Real^{k-1}$ denote the map
obtained from $\bf g$ by removing the $k$-th component.

The proof is by induction on $n$, with the base $n=1$ being trivial.
Choose any $b \in \Real$.
By the inductive hypothesis applied to the monotone map $\f_{i,j,b}$, where $i \neq k$, the map
$[\f_{i,j,b}]$ is monotone.
But $[\f_{i,j,b}]$ coincides with $[\f]_{i,j,b}$.
Hence the requirements (i) and (ii) in Definition~\ref{def:monot_map} are proved for $[\f]$.
\end{proof}

\begin{theorem}\label{th:monot-1}
The map $\f=(f_1):\> X \to \Real$ is monotone if and only if the function $f_1$ is monotone.
\end{theorem}

\begin{proof}
Suppose that $\f$ is a monotone map, let $\bf F$ be its graph.
If $f_1$ is a constant function, then it is trivially monotone.
Suppose that $f_1$ is non-constant.
Then, by item (i) of Definition~\ref{def:monot_map}, the item (ii) of Lemma~\ref{le:function_graph}
is satisfied.

It remains to show that the condition (i) of Lemma~\ref{le:function_graph} is also
valid for the function $f_1$.
We prove this by induction on $n$, the base for $n=1$ being a requirement in
Definition~\ref{def:monot_map}.

Suppose that for some $c_1, \ldots ,c_{n-1} \in \Real$ the set
$X \cap \{x_1=c_1, \ldots ,x_{n-1}=c_{n-1} \}$ is non-empty, and the restriction of $f_1$ to
this set is independent of $x_n$, i.e., identically equal to some $b \in \Real$.
We now prove that $f_1$ is independent of $x_n$,
i.e., remains a constant {\em for all} fixed values of $x_1, \ldots ,x_{n-1}$.

Since we assumed $f_1$ to be non-constant on $X$,
there exists $j$ such that $f_1$ is not independent of $x_j$.
As $\bf f$ is a monotone map, by item (i) in Definition~\ref{def:monot_map},
${\bf F}\cap\{y_1=b\}$ is a graph of a monotone map $\f_{1,j,b}=(f_{1,j,b})$
defined on a semi-monotone subset of ${\rm span} \{ x_1,\ldots,x_{j-1},x_{j+1},\ldots,x_n \}$.
Then $j\ne n$, since ${\bf F} \cap \{x_1=c_1,\ldots,x_{n-1}=c_{n-1},y_1=b\}$
contains a non-empty interval.

The function $f_{1,j,b}$ is monotone by the inductive hypothesis.
The restriction of $f_{1,j,b}$ on $\{x_1=c_1, \ldots ,x_{j-1}=c_{j-1},x_{j+1}=c_{j+1},
\ldots ,x_{n-1}=c_{n-1} \}$  is constant (identically equal to $c_j$).
Then, by item (i) of Lemma~\ref{le:function_graph}, $f_{1,j,b}$ is constant for all fixed values of
$x_1, \ldots , x_{j-1},x_{j+1}, \ldots ,x_{n-1}$.
Since, by item (ii) of Definition~\ref{def:monot_map}, the system of basis sets associated with
the map $\f_{1,j,b}$ does not depend on $b \in \Real$, the property of $f_{1,j,b}$ to be constant
for all fixed values of $x_1, \ldots , x_{j-1},x_{j+1}, \ldots ,x_{n-1}$ holds for all $b$.
Since the graph $\{ x_j=f_{1,j,b} \}$ of the function $f_{1,j,b}$ is the level set of $f_1$ at $b$,
and the union of all level sets for all values $b$ is $X$, it follows that for all fixed values of
variables $x_1, \ldots ,x_{n-1}$ the function $f_1$ is independent of $x_n$.

It follows that if the restriction of $f_1$ to $X \cap \{x_1=c_1, \ldots ,x_{n-1}=c_{n-1} \}$
is not independent of $x_n$ for some $c_1, \ldots ,c_{n-1} \in \Real$, then it is
not independent for all fixed values of $x_1, \ldots ,x_{n-1}$.
Repeating the argument from the proof of Lemma~\ref{le:function_basis} (for $j=n$), we conclude that
$f_1$ is either strictly increasing in, or strictly decreasing in, or independent of $x_n$.

Replacing in this argument $n$ by each of $1, \ldots ,n-1$, we conclude that
the item (i) of Lemma~\ref{le:function_graph} is satisfied, and therefore
the function $f_1$ is monotone.

Now suppose the function $f_1$ is monotone.
Then, by item (ii) in Lemma~\ref{le:function_graph}, the map $\f=(f_1)$ satisfies (i) in
Definition~\ref{def:monot_map}.
To prove that $\f$ satisfies also (ii) in Definition~\ref{def:monot_map}, fix the numbers
$b \in \Real$ and $j \in \{ 1, \ldots ,n \}$.
Suppose that $\{ x_1, \ldots ,x_{n-1} \}$ is a basis set of $\f_{1,j,b}$, i.e.,
the fibre $\{ x_1=c_1, \ldots ,x_{n-1}=c_{n-1}, f_1=b \}$, whenever non-empty,
is a single point for each sequence $c_1, \ldots c_{n-1} \in \Real$.
Then, according to (ii) in Definition~\ref{def:monotone}, $f_1$ is non-constant on
$\{ x_1=c_1, \ldots ,x_{n-1}=c_{n-1} \}$,
in particular, the fibre $\{ x_1=c_1, \ldots ,x_{n-1}=c_{n-1}, f_1=a \}$,
whenever non-empty, is a single point for any other value $a \in \Real$ of $f_1$.
It follows that the set $\{ x_1, \ldots ,x_{n-1} \}$ is a basis also for $\f_{1,j,a}$.

Replacing in this argument $n$ by each of $1, \ldots ,n-1$, we conclude that the {\em system}
of basis sets of $\f_{1,j,b}$ does not depend on $b$, hence $\f$ is monotone.
\end{proof}

\begin{corollary}\label{cor:monot-1}
If the map $\f=(f_1, \ldots ,f_k):\> X \to \Real^k$ is monotone, then every function $f_i$
is monotone.
\end{corollary}

\begin{proof}
Lemma~\ref{le:monot-1} implies that the map $(f_i):\> X \to \Real$ is monotone for every
$i=1, \ldots, k$.
Then, by Theorem~\ref{th:monot-1}, the function $f_i$ is monotone.
\end{proof}

\begin{remark}\label{re:non-monotone}
The converse to Corollary~\ref{cor:monot-1} is false when $n>1$ and $k>1$.
For example, consider the map $\f=(f_1,f_2):\> (\frac{1}{2},1)^2 \to \Real^2$, where
$$f_1=x_2/x_1\quad \text{and}\quad f_2=x_1-x_2.$$
Both functions, $f_1$ and $f_2$, are monotone on $(\frac{1}{2},1)^2$ but their level curves,
$\{ f_1=1 \}$ and $\{ f_2=0 \}$, coincide while all other pairs of level curves are different.
It follows that the map $\f_{1,1,1}$ has two basis sets, $\{ x_1 \}$ and $\{ x_2 \}$,
while $\f_{1,1,2}$ has three basis sets, $\{ x_1 \}$, $\{ x_2 \}$ and $\{ y_2 \}$.
Thus, the condition (ii) of Definition~\ref{def:monot_map} is not satisfied for $\f$.
\end{remark}

\begin{lemma}\label{le:fiber}
Let $\f:\> X \to \Real^k$ be a monotone map, ${\bf F}$ the graph of $\f$.
Then for any $\{ i_1, \ldots ,i_\beta \} \subset \{ 1, \ldots ,k \}$
and $b_1, \ldots ,b_\beta \in \Real$, where $\beta \le k$,
the intersection ${\bf F}_\beta:={\bf F} \cap \{ y_{i_1}=b_1, \ldots ,y_{i_\beta}=b_\beta \}$
is either empty or the graph of a monotone map defined on a semi-monotone set in
some space ${\rm span}\{ x_{j_1}, \ldots ,x_{j_\alpha} \}$, where $\alpha + \beta  \ge n$.
\end{lemma}

\begin{proof}
The proof is by induction on $\beta$.
If $\beta=0$ (the base of the induction), then ${\bf F}_0={\bf F}$ and hence is the graph of a
monotone map from $X \subset {\rm span} \{ x_1, \ldots ,x_n \}$ to ${\rm span} \{ y_1, \ldots ,y_k \}$.
Let $I_0= \emptyset$, and $J_0= \{ x_1, \ldots ,x_n \}$.

By the inductive hypothesis,
$${\bf F}_{\beta -1}:={\bf F} \cap \{ y_{i_1}=b_1, \ldots ,y_{i_{\beta -1}}=b_{\beta -1} \}$$
is a graph of a monotone map ${\bf h}=(h_1, \ldots ,h_k)$ from a semi-monotone subset of
${\rm span}\> J_{\beta-1}$ to
$${\rm span}((\{ y_1, \ldots ,y_k \} \setminus  I_{\beta-1})
\cup (\{ x_1, \ldots ,x_n \} \setminus  J_{\beta -1})).$$

If the function $h_{i_\beta}$ is constant in each of the variables in $J_{\beta -1}$, then
the graph ${\bf F}_{\beta -1}$ lies in $\{ y_{i_\beta}=c \}$ for some $c \in \Real$,
hence the intersection ${\bf F}_\beta={\bf F}_{\beta -1} \cap \{ y_{i_{\beta}}=b_{\beta} \}$
is either empty
(when $c \neq b_{\beta}$), or coincides with the graph ${\bf F}_{\beta -1}$ (when $c = b_{\beta}$).
In this case we consider ${\bf F}_\beta$ as the graph of the same map
${\bf h}$, and assume $I_\beta=I_{\beta -1},\> J_\beta=J_{\beta -1}$.

Suppose now that $h_{i_\beta}$ is not constant in some of $J_{\beta -1}$,
let it be, for definiteness, $x_{j_{\alpha +1}}$.
Let $I_\beta:= I_{\beta -1} \cup \{ y_{i_{\beta}} \}$ and
$J_\beta:= J_{\beta -1} \setminus \{ x_{j_{\alpha +1}} \}$.
Then, by Definition~\ref{def:monot_map},
${\bf F}_\beta={\bf F}_{\beta -1} \cap \{ y_{i_{\beta}}=b_{\beta} \}$
is the graph of the monotone map ${\bf h}_{i_{\beta}, x_{j_{\alpha +1}}, b_{\beta}}$
from a semi-monotone subset of ${\rm span}(J_\beta)$ to
$${\rm span}((\{ y_1, \ldots ,y_k \} \setminus  I_\beta)
\cup (\{ x_1, \ldots ,x_n \} \setminus  J_\beta)).$$
\end{proof}

\begin{notation}
For a subset $H \subset \{ x_1, \ldots ,x_n,y_1, \ldots ,y_k \}$, let ${\bf f}(H):\>X \to \Real^{|H|}$
denote a map defined by the functions corresponding to the elements of $H$.
\end{notation}

\begin{lemma}\label{le:non_constant}
If $H$ is a basis set of a monotone map $\f:\> X \to \Real^k$, then every component of $\f(H)$
is non-constant on $X$.
\end{lemma}

\begin{proof}
If all components of $\f(H)$ are coordinate functions, then this is obvious.
If non-coordinate functions exist and all are constants, then the dimension of each non-empty fibre
of ${\bf f}(H)$ equals to the number of these functions, i.e., greater than zero, which contradicts
to $H$ being a basis set.
Take a component $f_i$ of $\f(H)$ which is not a constant on $X$, thus it is non-constant in
some variable $x_j$.
Then each non-empty fibre of $f_i$ is the graph of a monotone map $\f_{i,j,b}$,
according to (i) in the Definition~\ref{def:monot_map}.
Observe that $H \setminus \{ y_i \}$ is a basis set for the map $\f_{i,j,b}$, since the fibres
of $\f_{i,j,b}(H \setminus \{ y_i \})$ are exactly those fibres of $\f (H)$ on which $f_i=b$.
If each component in $\f(H \setminus \{ y_i \})$ is constant on all $(n-1)$-dimensional fibres
of $f_i$, then we get a contradiction with $H \setminus \{ y_i \}$ being a basis for $\f_{i,j,b}$.
Continuing by induction, we conclude that all non-coordinate functions of $f(H)$ are
non-constant on $X$.
\end{proof}

\begin{theorem}\label{th:matroid}
Let $\f:\> X \to \Real^k$ be a monotone map on a non-empty semi-monotone $X \subset \Real^n$,
and $\bf F$ its graph.
Then
\begin{itemize}
\item[(i)]
The system $\bf m$ of basis sets associated with $\f$ is a matroid of rank $n$.
\item[(ii)]
For each independent set $I= \{ x_{j_1}, \ldots ,x_{j_\alpha}, y_{i_1}, \ldots ,y_{i_\beta} \}$ of
$\bf m$, and all sequences $c_1, \ldots ,c_\alpha,b_1, \ldots ,b_\beta \in \Real$, the non-empty
intersections
$${\bf F} \cap \{ x_{j_1}=c_1, \ldots ,x_{j_\alpha}=c_\alpha, y_{i_1}=b_1,
\ldots ,y_{i_\beta}=b_\beta \}$$
are graphs of  monotone maps,
having the same associated matroid ${\bf m}_I$
of rank $n - \alpha -\beta$ (the contraction of ${\bf m}$ by $I$).
In particular, all such intersections have the same dimension $n - \alpha - \beta$.
\end{itemize}
\end{theorem}

\begin{proof}
By the definition of a matroid, to prove (i), we need to check the {\em basis axiom}
(\cite{Welsh}, p. 8), which states
that for any two basis subsets $H,\> G \subset \{ x_1, \ldots ,x_n, y_1, \ldots ,y_k \}$,
if $h \in H \setminus G$ then there exists $g \in G \setminus H$ such that the set
$\{ g \} \cup (H \setminus \{ h \})$ is a basis set.
We prove this property by induction on $n$ simultaneously with the property (ii).
The base for $n=1$ is obvious for both (i) and (ii).

First we prove the inductive step for (i).
Fix any two basis subsets $H,\> G$, and an element $h \in (H \setminus G)$.
Consider the set $Z= H \setminus \{ h \}$.
We prove that each non-empty fibre of $\f (Z)$ is a graph of a univariate monotone map.
This is obvious if all components of $\f (Z)$ are coordinate functions.
If some non-coordinate functions exist then, by Lemma~\ref{le:non_constant}, they are non-constant.
Let $f_i$ be one of them.
In particular, $f_i$ is not independent of some variable $x_j$.
According to (i) in the Definition~\ref{def:monot_map}, each non-empty
set ${\bf F} \cap \{y_i=b \}$ is the graph of a monotone map $\f_{i,j,b}$.
Observe that $H \setminus \{ y_i \}$ is a basis set for the map $\f_{i,j,b}$ since the fibres
of $\f_{i,j,b}(H \setminus \{ y_i \})$ are exactly those fibres of $\f (H)$ on which $f_i=b$.
Note that the matroid ${\bf m}_i$ associated with $\f_{i,j,b}$ is the {\em contraction} of
the matroid $\bf m$ by $y_i$.
Since $H$ is a basis set for $\f$, all these fibres are single points, hence
$\f_{i,j,b}(H \setminus \{ y_i \})$ is injective.
Recall that $H \setminus \{ y_i \}=Z \setminus \{ h \}$.

By the inductive hypothesis of (ii), all non-empty fibres of $\f_{i,j,b}(H \setminus \{ y_i \})$
are one-dimensional graphs of monotone functions.
Since these fibres coincide with the fibres of $\f (Z)$ on which $f_i=b$, we conclude that all
non-empty fibres of $\f (Z)$ are one-dimensional graphs of monotone maps.

Now let $z \in \{ x_1, \ldots ,x_n, y_1, \ldots ,y_k \}$, and let $\psi$ be the corresponding
function.
Suppose that $Z \cup \{ z \}$ is not a basis set.
Then the restriction of $\psi$ to some fibre $\Psi$ of the map ${\f (Z)}$ has fibres of dimension
greater than zero.
But, as we proved above, all fibres of ${\f (Z)}$ are one-dimensional graphs of monotone maps,
hence $\psi$ is constant on $\Psi$.
Then item (ii) in the inductive hypothesis implies that $\psi$ is constant on each fibre of ${\f (Z)}$.

Suppose that the basis axiom is violated, i.e., given $h \in H$, for every $g \in G \setminus H$
the set $Z \cup \{ g \}$ is not a basis.
It follows that the function corresponding to every $g \in G \setminus H$
is constant on fibres of ${\f (Z)}$ (which are curves).
On the other hand, $G \cap H \subset Z$, so any function corresponding to $g \in G \cap H$
is constant on fibres of $\f (Z)$.
Therefore functions corresponding to all $g \in G$ are constant on fibres of $\f (Z)$,
hence fibres of $\f (G)$ are curves, thus $G$ is not a basis set, which is a contradiction.

Now we prove the inductive step of (ii).
It is sufficient to prove the statement for $|I|=1$ since the case of the general $I$ will
follow by induction on $|I|$.

If $I= \{ f_i \}$ for some $i=1, \ldots ,k$, then according to Lemma~\ref{le:non_constant},
$f_i$ is not a constant, and the statement follows immediately from Definition~\ref{def:monot_map}.

Now suppose that $I= \{ x_\ell \}$.

By Definition~\ref{def:monot_map} it is sufficient to prove that
\begin{itemize}
\item[(a)]
for all $i=1, \ldots k$, $j=1, \ldots ,n$ and $b \in \Real$, such
that $f_i$ is non-constant in $x_j$, the intersection ${\bf F} \cap \{ x_\ell=c \} \cap \{ y_i=b \}$
is the graph of a monotone map ${\bf g}_b$ on
a semi-monotone set in
${\rm span} \{ x_1, \ldots ,x_{\ell-1},x_{\ell+1}, \ldots ,x_{j-1},x_{j+1}, \ldots ,x_n \}$;
\item[(b)]
The matroid associated with ${\bf g}_b$ is the same for every $b \in \Real$.
\end{itemize}

By Definition~\ref{def:monot_map}, the set ${\bf F} \cap \{ y_i=b \}$ is the graph of a
monotone map $\f_{i,j,b}$.
Applying the inductive hypothesis to this monotone map we conclude that the set
${\bf F} \cap \{ y_i=b \} \cap \{ x_\ell=c \}$ is the graph of
a monotone map ${\bf g}_b$ on an $(n-2)$-dimensional semi-monotone set.
Hence, (a) is established.

By Definition~\ref{def:monot_map}, the maps $\f_{i,j,b}$ have the same system of basis sets
for all $b \in \Real$.
By the inductive hypothesis, this system is the matroid ${\bf m}_i$.
The common matroid for the maps ${\bf g}_b$ is obtained from ${\bf m}_i$ as follows.
Select all basis sets in ${\bf m}_i$ which contain the element $x_\ell$, and remove
this element from each of the selected sets.
The resulting system of sets forms the matroid for ${\bf g}_b$.
Note that this matroid is the {\em contraction} of ${\bf m}_i$ by $x_\ell$.
Since ${\bf m}_i$ is independent of $b$, so does this matroid, which proves (b).
\end{proof}

For any $I \subset \{x_1, \ldots ,x_n,y_1, \ldots ,y_k \}$, and ${\bf F} \subset \Real^{n+k}$,
let $T:= {\rm span}(I)$, and $\rho_T:\> {\bf F} \to T$ be the projection map.

\begin{theorem}\label{th:exchange}
Let ${\bf f}:\> X \to \Real^k$ be a monotone map on a non-empty semi-monotone $X \subset \Real^n$
having the graph ${\bf F} \subset \Real^{n+k}$.
Let ${\bf m}$ be the matroid associated with $\f$, $H$ a basis set of ${\bf m}$,
and $T:= {\rm span} (H)$.
Then $\rho_T({\bf F})$ is semi-monotone, and ${\bf F}$ is the graph of a monotone map on
$\rho_T({\bf F})$.
\end{theorem}

\begin{proof}
Suppose that $H=\{ x_1, \ldots ,x_{j-1},y_i,x_{j+1}, \ldots ,x_n \}$ for some $i$ and $j$.

Lemma~\ref{le:Usemi-monotone} implies that $\rho_T({\bf F})$ is a semi-monotone set.
Observe that ${\bf F}$ is the graph of a continuous map
$(f_1, \ldots ,f_{i-1},x_j,f_{i+1}, \ldots ,f_k)$ on $\rho_T({\bf F})$.
This map is monotone by the definition, since according to Theorem~\ref{th:matroid}, for each
$h \in \{ y_1, \ldots ,y_{i-1},x_j,y_{i+1}, \ldots ,y_k \}$ and for each $b \in \Real$,
the intersection ${\bf F} \cap \{ h=b \}$, if non-empty, is the graph of a monotone map with the
matroid independent of $b$.

For an arbitrary basis $H$, by the matroid's basis axiom, there exists a sequence
$$H_0:= \{ x_1, \ldots ,x_n \}, H_1, \ldots ,H_{t-1}, H_t=H$$
such that the basis $H_{\ell +1}$ is obtained from the basis $H_\ell$ by replacing a variable
of the kind $x_j$ by a variable of the kind $y_i$.
Applying the argument, described above to each such replacement, we conclude the proof.
\end{proof}

\begin{definition}\label{def:quasi-affine}
Let a bounded continuous map $\f=(f_1, \ldots ,f_k)$ defined on an open bounded non-empty set
$X \subset \Real^n$ have the graph ${\bf F} \subset \Real^{n+k}$.
We say that $\f$ is {\em quasi-affine} if for any
$T:= {\rm span} \{ x_{j_1}, \ldots , x_{j_\alpha}, y_{i_1}, \ldots ,y_{i_\beta}\}$,
where $\alpha + \beta=n$, the projection $\rho_T$ is injective if and only if the image
$\rho_T({\bf F})$ is $n$-dimensional.
\end{definition}

\begin{remark}\label{re:quasi-affine}
Observe that in Definition~\ref{def:quasi-affine} the property of $\rho_T$ to be injective
is equivalent to $\{ x_{j_1}, \ldots , x_{j_\alpha}, y_{i_1}, \ldots ,y_{i_\beta}\}$
being a basis set associated with $\f$.

In the case of a function, the property to be quasi-affine is equivalent
to the condition on the function to be either strictly increasing in, strictly decreasing
in, or independent of any variable.
\end{remark}

\begin{lemma}\label{le:tangent}
The system of basis sets associated with a quasi-affine map $\f:\> X \to \Real^k$, having
the graph $\bf F$, coincides with the system of bases of the matroid associated with the affine
map whose graph is the tangent space of $\bf F$ at a generic smooth point.
\end{lemma}

\begin{proof}
Follows immediately from Definitions~\ref{def:matroid} and \ref{def:quasi-affine}.

Observe that for some smooth points of $\bf F$ the matroids associated with the affine
maps, whose graphs are tangent spaces of $\bf F$ at these points, may be different from
the matroid at a generic smooth point (e.g., the origin for $\f(x)=x^3$ on $X=(-1,1)$).
\end{proof}

\begin{theorem}\label{th:vdd}
Every monotone map $\f:\> X \to \Real^k$ is quasi-affine.
\end{theorem}

\begin{proof}
Let $\f:\> X \to \Real^k$ be a monotone map having the graph ${\bf F} \subset \Real^{n+k}$,
and $T:= {\rm span} \{ x_{j_1}, \ldots , x_{j_\alpha}, y_{i_1}, \ldots ,y_{i_\beta}\}$,
where $\alpha + \beta=n$.

If the projection $\rho_T$ is injective, then obviously $\rho_T({\bf F})$ is $n$-dimensional.

Conversely, observe that, by Lemma~\ref{le:monot-1}, the map
${\bf g}:=(f_{i_1}, \ldots ,f_{i_\beta})$ is monotone.
Suppose that, contrary to the claim, $\rho_T$ is not injective.
Lemma~\ref{le:fiber} implies that if the fiber of a monotone map over a value is
0-dimensional then it is a single point.
Hence, there are two points
$$(a_{j_1}, \ldots ,a_{j_\alpha},a_{i_1}, \ldots ,a_{i_\beta})\> \text{and}\>
(b_{j_1}, \ldots ,b_{j_\alpha},b_{i_1}, \ldots ,b_{i_\beta})$$
in ${\rm span} \{ x_{j_1}, \ldots , x_{j_\alpha}, y_{i_1}, \ldots ,y_{i_\beta}\}$
such that the fiber of ${\bf g}':={\bf g}|_{x_{j_1}=a_{j_1}, \ldots , x_{j_\alpha}=a_{j_\alpha}}$
over $(a_{i_1}, \ldots ,a_{i_\beta})$ is a single point,
while the fiber of ${\bf g}'':={\bf g}|_{x_{j_1}=b_{j_1}, \ldots , x_{j_\alpha}=b_{j_\alpha}}$
over $(b_{i_1}, \ldots ,b_{i_\beta})$ has the positive dimension.
By the part (ii) of Theorem~\ref{th:matroid}, applied to the independent set
$\{ x_{j_1}, \ldots , x_{j_\alpha} \}$ as $I$, the matroids, associated with ${\bf g}'$
and ${\bf g}''$ coincide.
In particular, there exists a point $(b^{(0)}_{i_1}, \ldots ,b^{(0)}_{i_\beta})$ such that
the fiber of ${\bf g}'$ over this point has the positive dimension.
Let ${\bf G}_0$ be the graph of ${\bf g}'$.

We proceed by induction on $\nu=0,1, \ldots ,\beta$.
According to Lemma~\ref{le:fiber}, for every $\nu \le \beta$ the intersection
${\bf G}_\nu:={\bf G}_0 \cap \{ y_{i_1}=a_{i_1} \ldots , y_{i_\nu}=a_{i_\nu} \}$
is the graph of a monotone map ${\bf g}^{(\nu)}$
(since the fiber of ${\bf g}'$ over $a_{i_1}, \ldots ,a_{i_\beta}$ is 0-dimensional, the set
${\bf G}_\nu$ is non-empty).
Also, since the fiber of ${\bf g}'$ is 0-dimensional, the map ${\bf g}^{(\nu)}$ is of the form
${\bf g}^{(\nu -1)}_{i_\nu,j,a_{i_\nu}}$ for an appropriate $j$.
Because ${\bf g}^{(\nu-1)}$ is monotone, the map ${\bf g}^{(\nu -1)}_{i_\nu,j,b^{(\nu -1)}_{i_\nu}}$
has the same matroid as ${\bf g}^{(\nu -1)}_{i_\nu,j,a_{i_\nu}}$.
In particular, there exists a point $(b^{(\nu)}_{i_{\nu +1}}, \ldots ,b^{(\nu)}_{i_\beta})$
such that the fiber of ${\bf g}^{(\nu -1)}_{i_\nu,j,a_{i_\nu}}$ over this point has
the positive dimension.

On the last step, for $\nu = \beta -1$ and an appropriate $j$, the two maps,
${\bf g}^{(\beta -1)}={\bf g}^{(\beta -2)}_{i_{\beta -1},j,a_{i_{\beta -1}}}$
and ${\bf g}^{(\beta -2)}_{i_{\beta -1},j,b^{(\beta -2)}_{i_{\beta -1}}}$, defined on an interval,
have the same matroids.
On the other hand, the component of ${\bf g}^{(\beta -2)}_{i_{\beta -1},j,a_{i_{\beta -1}}}$,
corresponding to $y_{i_\beta}$, is a non-constant monotone function on that interval,
while the component of ${\bf g}^{(\beta -2)}_{i_{\beta -1},j,b^{(\beta -2)}_{i_{\beta -1}}}$,
corresponding to $y_{i_\beta}$, is a constant function.
We get a contradiction.
\end{proof}

\begin{corollary}\label{cor:independent}
Let a monotone map $\f:\> X \to \Real^k$ have the graph ${\bf F} \subset \Real^{n+k}$
and the associated matroid $\bf m$.
A subset $H \subset \{ x_1, \ldots ,x_n,y_1, \ldots ,y_k \}$ is an independent set of $\bf m$
if and only if $\dim (\rho_L ({\bf F}))= |H|$, where $L:= {\rm span}\> H$.
\end{corollary}

\begin{proof}
Let $|H|=m$.
If $H$ is an independent set of $\bf m$ then, by the matroid theory's
Augmentation Theorem (\cite{Welsh}, Ch.~1,
Section~5), there is a basis set $I$ of $\bf m$ such that $H \subset I$.
By Theorem~\ref{th:vdd}, $\dim (\rho_T ({\bf F}))=n$, where $T:= {\rm span}\> I$, and therefore
$\dim (\rho_L ({\bf F}))= m$, since $\rho_L ({\bf F})$ is the image of the projection of
$\rho_T ({\bf F})$ to $L$.

Conversely, suppose that $\dim (\rho_L ({\bf F}))= m$.
Clearly, $m \le n$.
Observe that for any definable pure-dimensional set $U$ in $\Real^{n+k}$, with $\dim U=n$, if
$\rho_L (U)=m$ then there is a subset $I \subset \{ x_1, \ldots ,x_n,y_1, \ldots ,y_k \}$ such
that $H \subset I$ and $\dim (\rho_T (U))=n$, where $T= {\rm span}\> I$.
Since $\bf F$, being the graph of continuous map on an open set in $\Real^n$, satisfies all the
properties of $U$, there exists a subset $I$ such that $\dim (\rho_T ({\bf F}))=n$.
Then $H$ is an independent set of $\bf m$, being a subset of its basis.
\end{proof}

\begin{remark}\label{re:tengent}
Lemma~\ref{le:tangent} immediately implies that the matroid associated with a monotone map, having
the graph $\bf F$, coincides with the matroid associated with the affine map whose graph
is the tangent space of $\bf F$ at a generic point.
\end{remark}

\section{Equivalent definitions of a monotone map,
and their corollaries}\label{sec:equivalent_defs_of_monotone_maps}

\begin{lemma}\label{le:sign}
Let $\f:\> X \to \Real^k$ be a monotone map on a semi-monotone set $X \subset \Real^n$, and $C$
a coordinate cone in $\Real^n$.
Then the restriction $\f|_C:\> X \cap C \to \Real^k$ is a monotone map.
\end{lemma}

\begin{proof}
It is sufficient to consider just the cases of $C= \{ x_\ell =c \}$ and $C= \{ x_\ell >c \}$.
The first case follows directly from Theorem~\ref{th:matroid}, (ii).
The proof for the case $C= \{ x_\ell >c \}$ can be conducted by induction on $n$, completely analogous
to the last part of the proof of Theorem~\ref{th:matroid}.
\end{proof}

\begin{corollary}\label{cor:monot_in_box}
Let ${\bf F} \subset \Real^{n+k}$ be the graph of a monotone map $\f: X \to \Real^k$, and let
$$P := \bigcap_{1 \le j \le n+k} \{(z_1, \ldots ,z_{n+k})|\> a_j<z_j < b_j \}
\subset \Real^{n+k}$$
for some $a_j,\> b_j \in \Real,\> j= 1, \ldots ,n+k$.
Then ${\bf F} \cap P$ is either empty or the graph of a monotone function.
\end{corollary}

\begin{proof}
Let $P'$ be the image of the projection of $P$ to $X$.
By Lemma~\ref{le:sign}, the restriction $\f_{P'}:\> P' \to \Real^k$ is a monotone map.
Theorem~\ref{th:exchange} allows to apply the same argument to projections of $P$ to other
subspaces of $\Real^{n+k}$.
\end{proof}

The following theorem is a generalization of Lemma~\ref{le:def_monotone}
from monotone functions to monotone maps.

\begin{theorem}\label{th:def_monotone_map}
Let a bounded continuous quasi-affine map $\f=(f_1, \ldots ,f_k)$ defined on an open bounded
non-empty set $X \subset \Real^n$ have the graph ${\bf F} \subset \Real^{n+k}$.
The following three statements are equivalent.

\begin{enumerate}
\item[(i)]
The map $\f$ is monotone.
\item[(ii)]
For each affine coordinate subspace $S$ in $\Real^{n+k}$ the intersection ${\bf F} \cap S$ is connected.
\item[(iii)]
For each coordinate cone $C$ in $\Real^{n+k}$ the intersection ${\bf F} \cap C$ is connected.
\end{enumerate}
\end{theorem}

\begin{proof}
We first prove that (i) implies (ii).
Suppose that $\f$ is monotone.
Consider an affine coordinate subspace $S= \{ x_{j_1}=c_1, \ldots ,x_{j_\alpha}=c_\alpha,
y_{i_1}=b_1, \ldots ,y_{i_\beta}=b_\beta \}$ in $\Real^{n+k}$.
Lemma~\ref{le:sign} implies that the intersection
${\bf F} \cap \{ x_{j_1}=c_1, \ldots ,x_{j_\alpha}=c_\alpha \}$ is the graph of a monotone
map ${\bf g}$, and hence connected.
Applying Lemma~\ref{le:fiber} to ${\bf g}$, we conclude that ${\bf F} \cap S$ is also the graph of
a monotone map, and therefore connected.

The part (ii) implies the part (iii) by Lemma~\ref{le:connected} and
a straightforward induction on the number of strict inequalities defining the coordinate cone.

Now we prove that (iii) implies (i).
Suppose the condition (iii) is satisfied.
Let $C'$ be a coordinate cone in ${\rm span} \{ x_1, \ldots ,x_n \}$.
Then the intersection ${\bf F} \cap (C' \times \Real^k)$ is connected, hence the image of its
projection, $C' \cap X$, is connected.
It follows that $X$ is semi-monotone.

We prove that $\f$ is a monotone map by induction on $n$, the base for $n=1$ being trivial.
Choose any $i$ and $j$ such that $f_i$ is non-constant in $x_j$, and consider the map ${\bf f}_{i,j,b}$
for some $b \in \Real$.
This map and its graph ${\bf F} \cap \{y_i=b \}$ inherit from $\f$ and ${\bf F}$ properties
(i) and (ii) in the conditions of the theorem.
Thus, by the inductive hypothesis, ${\bf f}_{i,j,b}$ is a monotone map.
It remains to prove that the system of basis sets associated with ${\bf f}_{i,j,b}$ does not depend
on $b \in \Real$.
Let $T={\rm span} \{ x_1, \ldots , x_{j-1},x_{j+1}, \ldots ,x_n \}$.
Any basis set of ${\bf f}_{i,j,b}$ is a subset
$$\{ x_{j_1}, \ldots , x_{j_\alpha}, y_{i_1}, \ldots ,y_{i_\beta}\} \subset
\{ x_1, \ldots ,x_n,y_1, \ldots , y_{i-1}, y_{i+1},\ldots , y_k \},$$
where $\alpha+ \beta=n-1$, such that the map
$$(x_{j_1}, \ldots , x_{j_\alpha}, f_{i_1}, \ldots ,f_{i_\beta}):\> \rho_T({\bf F}
\cap \{ y_i=b \}) \to \Real^{n-1}$$
is injective.
Since $\f$ is quasi-affine, the injectivity does not depend on the choice of $b$.
\end{proof}

\begin{corollary}
Let $\f:\> X \to \Real^k$ be a monotone map having the graph ${\bf F} \subset \Real^{n+k}$.
Then for every $z \in \{ x_1, \ldots ,x_n,y_1, \ldots ,y_k \}$ and every $c \in \Real$
each of the intersections ${\bf F} \cap \{ z\> \sigma\> c \}$, where $\sigma \in \{ <,>,= \}$,
is either empty or the graph of a monotone map.
\end{corollary}

\begin{proof}
%Consider, for example the case of a non-empty set ${\bf F} \cap \{ y_i < b \}$
%(all other cases are similar).
Let ${\bf F} \cap \{ z\> \sigma\> c \} \neq \emptyset$.
According to part (iii) of Theorem~\ref{th:def_monotone_map}, for any affine coordinate subspace
$S \subset \Real^{n+k}$, the intersection ${\bf F} \cap \{ z\> \sigma\> c \} \cap S$ is connected,
since $\{ z\> \sigma\> c \} \cap S$ is a coordinate cone.
Then part (ii) of this theorem implies that ${\bf F} \cap \{ z\> \sigma\> c \}$ is
the graph of a monotone map.
\end{proof}

The following corollary is a generalization of Definition~\ref{def:monotone}
from monotone functions to monotone maps.

\begin{corollary}\label{cor:map-function}
Let a bounded continuous map $\f=(f_1, \ldots ,f_k)$ on a non-empty semi-monotone set
$X \subset \Real^n$ have the graph ${\bf F} \subset \Real^{n+k}$.
The map $\f$ is monotone if and only if
\begin{enumerate}
\item[(i)]
it is quasi-affine;
\item[(ii)]
for every subset $\{ i_1, \ldots ,i_m \} \subset \{ 1, \ldots ,k \}$ the intersection
$$S_{i_1, \ldots ,i_m}:= \bigcap_{1 \le \ell \le m}\{ f_{i_\ell} \sigma_\ell 0 \} \subset X,$$
where $\sigma _\ell \in \{ <,> \}$, is semi-monotone.
\end{enumerate}
\end{corollary}

\begin{proof}
Suppose that the map $\f$ is monotone.
Then, by Theorem~\ref{th:def_monotone_map}, the condition (i) is satisfied.
Let $C$ be a coordinate cone in ${\rm span}\{ x_1, \ldots , x_n \}$.
The intersection $S_{i_1, \ldots ,i_m} \cap C$ is the projection on
${\rm span} \{ x_1, \ldots , x_n \}$ of the intersection of ${\bf F}$ with the coordinate cone
$$K:= (C \times \Real^k) \cap \bigcap_{1 \le \ell \le m}\{ y_{i_\ell} \sigma_\ell 0 \}$$
in ${\rm span} \{ x_1, \ldots , x_n,y_1, \ldots ,y_k \}$.
By item (ii) in Theorem~\ref{th:def_monotone_map}, this intersection is connected, hence
the projection is also connected.
It follows that $S_{i_1, \ldots ,i_m}$ is semi-monotone.

Conversely, suppose that the conditions (i) and (ii) of the theorem are satisfied, and let $K$
be a coordinate cone in ${\rm span} \{ x_1, \ldots , x_n,y_1, \ldots ,y_k \}$.
The projection of ${\bf F} \cap K$ on ${\rm span} \{ x_1, \ldots , x_n \}$ is of the kind
$S_{i_1, \ldots ,i_m} \cap C$ for some $\{ i_1, \ldots ,i_m \} \subset \{ 1, \ldots ,k \}$
and a coordinate cone  $C$ in ${\rm span} \{ x_1, \ldots , x_n \}$.
Since $S_{i_1, \ldots ,i_m}$ is semi-monotone, its intersection with $C$ is connected.
Because ${\bf F}$ is the graph of a continuous map, the intersection ${\bf F} \cap K$ is also
connected, hence, by Theorem~\ref{th:def_monotone_map}, the map $\f$ is monotone.
\end{proof}

The following theorem is another corollary to Theorem~\ref{th:def_monotone_map}

\begin{theorem}\label{th:proj_graph}
Let $\f:\> X \to \Real^k$ be a monotone map defined on a semi-monotone set $X \subset \Real^n$,
with the associated matroid ${\bf m}$, and graph ${\bf F} \subset \Real^{n+k}$.
Then for any subset $I \subset \{ x_1, \ldots , x_n,y_1, \ldots ,y_k \}$, with
$T:= {\rm span}\> I$, the image $\rho_T ({\bf F})$ under the projection map $\rho_T:\> {\bf F} \to T$
is either a semi-monotone set or the graph of a monotone map, whose matroid is a minor of ${\bf m}$.
\end{theorem}

\begin{proof}
By Theorem~\ref{th:def_monotone_map}, it is sufficient to prove that ${\bf G}:=\rho_T ({\bf F})$ is
either a semi-monotone set or the
graph of a quasi-affine map, and that for each affine coordinate subspace $S$ in ${\rm span}\> I$
the intersection $\rho_T ({\bf F}) \cap S$ is connected.

Let $\dim {\bf G}=m$ and assume first that $m< \dim T$.

Let $H$ be a subset of $I$ such that $|H|=m$ and $\dim (\lambda_L ({\bf G}))=m$,
where $L:= {\rm span}\> H$, and $\lambda_L:\> {\bf G} \to L$ is the projection map
(obviously, there is such a subspace).
By Corollary~\ref{cor:independent},
there exists a subset $J \subset \{ x_1, \ldots , x_n,y_1, \ldots ,y_k \}$ such that
$J \cap I=H$, $|J|=n$, and $\dim (\rho_N ({\bf F}))=n$, where $N:= {\rm span}\> J$.

Since $\f$ is quasi-affine, the projection map $\rho_N:\> {\bf F} \to N$ is injective.
We now prove that the projection map $\lambda_L$ is also injective.
Because $\rho_N$ is injective, the set $J$ is a basis of ${\bf m}$, and therefore its subset
$H$ is an independent set of ${\bf m}$.
According to Theorem~\ref{th:matroid}, all non-empty fibers of $\rho_L:\> {\bf F} \to L$ are
graphs of monotone maps, having the same matroids of rank $n-m$.
Since $\dim {\bf G}=m$, the image of the projection to $T$ of a generic fiber of $\rho_L$
(hence every fiber of $\rho_L$, since all matroids are the same) to $T$ is 0-dimensional, thus
it is a single point.
We conclude that $\lambda_L$ is injective, hence ${\bf G}$ is the graph of a map,
defined on $\lambda_L ({\bf G})$, and that this map is quasi-affine.
Denote this map by ${\bf g}$.

Let $S$ be an affine coordinate subspace in $T$.
Since $\f$ is monotone, the intersection ${\bf F} \cap (S \times ({\rm span}\> (J \setminus H)))$
is connected.
It follows that ${\bf G} \cap S= \rho_L ({\bf F} \cap (S \times ({\rm span}\> (J \setminus H)))$
is connected, hence, by Theorem~\ref{th:def_monotone_map}, ${\bf g}$ is a monotone map.

Let ${\bf m}_I$ be the matroid associated with ${\bf g}$.
Observe that if the projection map $\rho_T$ is injective, then the family of all independent sets
of ${\bf m}_I$ consists of subsets $H \subset I$ which are independent sets of ${\bf m}$.
It follows that ${\bf m}_I$ is a {\em restriction} of the matroid $\bf m$ to $I$
(\cite{Welsh}, Ch.~4, Section~2).
In fact, ${\bf m}_I$ includes some basis sets of $\bf m$.
In the case of a general projection map, the family of all independent sets of $\bf m$ consists
of subsets $H \subset I$ which can be appended by maximal independent subsets of $J \setminus H$
so that to become independent sets of $\bf m$.
Thus, ${\bf m}_I$ is a {\em contraction} of $\bf m$ (\cite{Welsh}, Ch.~4 Section~3).
In fact, for each basis set of ${\bf m}_I$ there is a subset of $J \setminus H$ such that
their union is a basis of $\bf m$.
It follows that ${\bf m}_I$ is a {\em minor} of $\bf m$.

Now let $\dim {\bf G}=m= \dim T$.
By Corollary~\ref{cor:independent}, $I$ is an independent set of the matroid $\bf m$,
hence, by the matroid theory's Augmentation Theorem (\cite{Welsh}, Ch.~1, Section~5), there is a basis
$J$ of $\bf m$, containing $I$.
The image of the projection of $\bf F$ to ${\rm span}\> J$ is a semi-monotone set, according to
Theorems~\ref{th:vdd} and \ref{th:exchange}.
Then, by Proposition~\ref{prop:proj}, $\bf G$ is also a semi-monotone set.
\end{proof}

\begin{theorem}\label{th:schonflies_for_monotone}
Let $\bf F$ be the graph of a monotone map $\f:\> X \to \Real^k$ on a semi-monotone set
$X \subset \Real^n$.
Let ${\bf G} \subset {\bf F}$ be the graph of a monotone map ${\bf g}$ such that
$\dim {\bf G} =n-1$, and $\partial {\bf G} \subset \partial{\bf F}$.
Then ${\bf F} \setminus {\bf G}$ is a disjoint union of two graphs of some monotone maps.
\end{theorem}

\begin{proof}
Let $T:= {\rm span} \{ x_1, \ldots ,x_n \}$.
According to Theorem~\ref{th:proj_graph}, $\rho_T({\bf G})$ is a graph of a monotone function,
hence, by Lemma~\ref{le:divide}, $X \setminus \rho_T({\bf G})$ has two connected components,
each of which is a semi-monotone set.
Their pre-images in $\bf F$ are the two connected components of ${\bf F} \setminus {\bf G}$,
we denote them by ${\bf F}_+$ and ${\bf F}_-$, which are graphs of some continuous maps,
$\f_+$ and $\f_-$.
We prove that $\f_+$ and $\f_-$ are monotone maps by checking that they satisfy
Definition~\ref{def:monot_map}.

Suppose that there exist $i \in \{ 1, \ldots ,k \}$ and $j \in \{ 1, \ldots ,n \}$ such that
the component $f_i$ of $\f$ is not independent of a coordinate $x_j$,
otherwise $\f$ is identically constant on $X$, and the theorem becomes trivially true.
Fix one such pair $i,j$.

The intersection ${\bf F} \cap \{ y_i=b \}$ is either empty or the graph of a monotone map,
due to Definition~\ref{def:monot_map}.
If ${\bf G} \subset {\bf F} \cap \{ y_i=b \}$, then ${\bf G} = {\bf F} \cap \{ y_i=b \}$, and hence
$({\bf F}_+ \cup {\bf F}_-) \cap \{ y_i=b \}=\emptyset$.
We now prove, by induction on $n$, that if ${\bf F} \cap \{ y_i=b \} \neq \emptyset$ and
${\bf G} \not\subset {\bf F} \cap \{ y_i=b \}$, then each of intersections ${\bf F}_+ \cap \{ y_i=b \}$
and ${\bf F}_- \cap \{ y_i=b \}$ is either empty or the graph of a monotone map.
The base of the induction, for $n=1$, is trivial.

If ${\bf G} \cap \{ y_i=b \} =\emptyset$, then either
${\bf F}_+ \cap \{ y_i=b \}={\bf F} \cap \{ y_i=b \}$ or
${\bf F}_- \cap \{ y_i=b \}={\bf F} \cap \{ y_i=b \}$.
In any case, one of the two intersections is the graph of a monotone map and the other one is empty.
Assume now that ${\bf G} \cap \{ y_i=b \} \neq \emptyset$.
Then both intersections, ${\bf F}_+ \cap \{ y_i=b \}$ and ${\bf F}_- \cap \{ y_i=b \}$ are non-empty.
Indeed, if one of them is empty, then the component $g_i$ of ${\bf g}$
attains the global extremum $b$ at some point in its domain, which contradicts the monotonicity of $g_i$.

By the inductive hypothesis, $({\bf F} \cap \{y_i =b \}) \setminus ({\bf G} \cap \{y_i =b \})$
is a disjoint union of two graphs of some monotone maps.
One of its connected components lies in ${\bf F}_+$ while another in ${\bf F}_-$.
Hence, both intersections ${\bf F}_+ \cap \{ y_i=b \}$ and ${\bf F}_- \cap \{ y_i=b \}$
are graphs of monotone maps, for example, of the maps $\f_{+\ i,j,b}$ and $\f_{-\ i,j,b}$.
Thus, the part (i) of Definition~\ref{def:monot_map} is proved for $\f_+$ and $\f_-$.

Observe that the matroid associated with the monotone map $\f_{i,j,b}$ does not depend on
$b$ (by the part (ii) of Definition~\ref{def:monot_map}), and, since $\f_{i,j,b}$ is quasi-affine (by
Theorem~\ref{th:vdd}), coincides with systems of basis sets
of each of the maps $\f_{+\ i,j,b}$ and $\f_{-\ i,j,b}$.
It follows that the systems of basis sets $\f_{+\ i,j,b}$ and $\f_{-\ i,j,b}$ do not depend on $b$,
which proves the part (ii) of Definition~\ref{def:monot_map} for $\f_+$ and $\f_-$.

We conclude that the maps $\f_+$ and $\f_-$ are monotone.
\end{proof}

\begin{definition}\label{def:sign}
Let $g_1, \ldots ,g_m$ be continuous functions on a set $Y \subset \Real^n$.
A {\em sign condition in $g_1, \ldots ,g_m$ on $Y$} is any intersection of the kind
$Y \cap \bigcap_{1 \le i \le m} \{ g_i \sigma_i 0 \}$,
where $\sigma_i \in \{ <, =, > \}$.
\end{definition}

\begin{corollary}
Let $\bf F$ be the graph of a monotone map $\f:\> X \to \Real^k$ on a semi-monotone set
$X \subset \Real^n$.
Let $g_1, \ldots , g_m$ be continuous functions $\bf F:\> \to \Real$ such that
for any subset $\{ i_1, \ldots ,i_\ell \} \subset \{ 1, \ldots ,m \}$ and any
$j>1$, the intersection ${\bf F} \cap \{ g_{i_1}= \cdots =g_{i_j}=0 \}$ either coincides with
the intersection ${\bf F} \cap \{ g_{i_1}= \cdots =g_{i_{j-1}}=0 \}$ or is the graph of a monotone map
having codimension 1 in $\bf F$.
Then every sign condition in $g_1, \ldots , g_m$ on $\bf F$ is the graph of a monotone function.
\end{corollary}

\begin{proof}
Given a sign condition in $g_1, \ldots , g_m$ on $\bf F$, assume, without loss of generality,
that all equalities among $\sigma_i$ correspond to $i=1, \ldots ,r$, where $r \le m$.
Then, by the conditions of the lemma, ${\bf F} \cap \{ g_1= \cdots =g_r=0 \}$
is the graph of a monotone map.
Adding the strict inequalities $g_j \sigma_j 0$ for $j>r$ one by one we conclude, by
Theorem~\ref{th:schonflies_for_monotone}, that the set in $\bf F$, defined by equations and inequalities
on each step, is the graph of a monotone map.
\end{proof}

The following theorem is a version for monotone maps of Theorem~\ref{th:monot_funct}.

\begin{theorem}\label{th:monot_maps}
Let a bounded continuous quasi-affine map $\f=(f_1, \ldots ,f_k)$ on a non-empty semi-monotone set
$X \subset \Real^n$ have the graph ${\bf F} \subset \Real^{n+k}$.
Let
$$\{ x_{j_1}, \ldots , x_{j_\alpha}, y_{i_1}, \ldots ,y_{i_\beta}\} \subset
\{ x_1, \ldots ,x_n,y_1, \ldots , y_k \},$$
where $\alpha+ \beta=k$, and
$$T:= {\rm span} (\{ x_1, \ldots ,x_n,y_1, \ldots , y_k \} \setminus
\{ x_{j_1}, \ldots , x_{j_\alpha}, y_{i_1}, \ldots ,y_{i_\beta}\}).$$
Assume that $\rho_T({\bf F})$ is $n$-dimensional.
The map $\f$ is monotone if and only if
\begin{itemize}
\item[(i)]
The image $\rho_T({\bf F})$ is a semi-monotone set.
\item[(ii)]
For every $j \in \{ j_1, \ldots ,j_\alpha \}$ and every $i \in \{ i_1, \ldots ,i_\beta \}$,
such that the function $f_i$ is not independent of a variable form
$\{ x_1, \ldots ,x_n\} \setminus \{ x_{j_1}, \ldots , x_{j_\alpha}\}$,
all non-empty intersections ${\bf F} \cap \{ x_j=a \}$ and ${\bf F} \cap \{ y_i=a \}$,
where $a \in \Real$, are either empty or graphs of monotone maps.
\end{itemize}
\end{theorem}

\begin{proof}
Suppose that $\f$ is a monotone map.
Then (i) and (ii) are satisfied by Theorem~\ref{th:exchange}, because $\f$ is quasi-affine.

Conversely, assume that a bounded continuous quasi-affine map $\f$ satisfies the properties (i) and (ii).
Because $\f$ is quasi-affine and (i) is satisfied, ${\bf F}$ is the graph of a map
${\bf g}:\> \rho_T({\bf F}) \to {\rm span} \{ x_{j_1}, \ldots ,
x_{j_\alpha}, y_{i_1}, \ldots ,y_{i_\beta}\}$.
Again, since $\f$ is quasi-affine, all monotone (according to (ii)) maps, with graphs
${\bf F} \cap \{ x_j=a \}$ and ${\bf F} \cap \{ y_i=a \}$, have associated matroids that
don't depend on $a$.
It follows from Definition~\ref{def:monot_map} that ${\bf g}$ is monotone.
Since $\f$ and ${\bf g}$ have the same graph, Theorem~\ref{th:def_monotone_map} implies that
that $\f$ is also monotone.
\end{proof}

\section{Graphs of monotone maps are regular cells}\label{sec:regular}

It is known (see Proposition~\ref{pr:triang}) that any compact definable set $X \subset \Real^n$
is definably homeomorphic to a finite simplicial complex
$\widetilde X$, which is a {\em polyhedron} \cite{RS}.
In this section we will use Lemma~\ref{le:pl_hom} and some known results from PL topology
(formulated, for the reader's convenience, in Section~\ref{sec:appendix}) to introduce
or to prove some definable homeomorphisms of definable sets.
Thus, the relation $X \sim Y$ (``$X$ is definably homeomorphic to $Y$'') we will understand
as $\widetilde X \sim_{PL}\> \widetilde Y$  (``$\widetilde X$ is PL homeomorphic to $\widetilde Y$'').

Throughout this section the term ``regular cell'' means ``topologically regular cell''.
In line with the convention above, we will actually use a slightly stronger
version of this notion than the one in Definition~\ref{def:regular}.
Namely, we say that a definable set $V$ is a closed $n$-ball if $\widetilde V \sim_{PL}\> [-1, 1]^n$,
and is an $(n-1)$-sphere if $\widetilde V \sim_{PL}\> ([-1,1]^n \setminus (-1, 1)^n)$.
A definable bounded open set $U \subset \Real^n$ is called (topologically) regular cell
if $\overline U$ is a closed ball,
and the frontier $\overline U \setminus U$ is an $(n-1)$-sphere.
By Proposition~\ref{pr:le1.10}, such $U$ is also a regular cell in the sense of
Definition~\ref{def:regular}.

\begin{theorem}\label{th:regularcell}
The graph ${\bf F} \subset \Real^{n+k}$ of a monotone map $\f: X \to \Real^k$ on a semi-monotone set
$X \subset {\rm span} \{ x_1, \ldots ,x_n \}$ is a regular $n$-cell.
\end{theorem}

\begin{remark}
Theorem~2.2 in \cite{BGV} states that every semi-monotone set is a regular cell.
Example~4.3 in \cite{BGV} shows that the graph a bounded submonotone (but not supermonotone) function,
satisfying the condition (ii) in Definition~\ref{def:monotone}, may not be a regular cell.
\end{remark}

We are going to prove Theorem \ref{th:regularcell}
by induction on the dimension $n$ of a regular cell.
For $n=1$ the statement is obvious.
Assume it to be true for $n-1$, we will refer to this statement as to the {\em global inductive
hypothesis}.

\begin{lemma}\label{cut}
Let ${\bf F} \subset \Real^{n+k}$ be a graph of a monotone map.
Let
$${\bf F}_0:={\bf F} \cap X_{j,=,c},\quad {\bf F}_+ :={\bf F} \cap X_{j,>,c},\quad
\text{and}\quad {\bf F}_- :={\bf F} \cap X_{j,<,c}$$
for some $1 \le j \le n+k$ and $c \in \Real$.
Then $\overline {\bf F}_+\cap \overline {\bf F}_-=\overline {\bf F}_0$.
\end{lemma}

\begin{proof}
Let a point $\x=(x_1, \ldots ,x_n) \in X_{j,=,c} \setminus \overline {\bf F}_0$
belong to $\overline {\bf F}_+\cap \overline {\bf F}_-$.
Then there is an $\eps >0$ such that an open cube centered at $\x$,
$$P_\eps := \bigcap_{1 \le j \le n+k} \{(y_1, \ldots ,y_{n+k})|\> |x_j - y_j| < \eps \}
\subset \Real^{n+k},$$
has non-empty intersections with both ${\bf F}_+$ and ${\bf F}_-$
and the empty intersection with ${\bf F}_0$.
Thus, $P_\eps \cap {\bf F}$ is not connected, which is not possible since, according to
Corollary~\ref{cor:monot_in_box}, $P_\eps \cap {\bf F}$ is the graph of a monotone map.
\end{proof}

\begin{corollary}\label{union}
Let ${\bf F} \subset \Real^{n+k}$ be a graph of a monotone map.
If ${\bf F}_+$ and ${\bf F}_-$ in Lemma \ref{cut} are regular cells, then ${\bf F}$ is a regular cell.
\end{corollary}

\begin{proof}
We need to prove that $\overline {\bf F}$ is a closed $n$-ball, and that the frontier
$\overline {\bf F} \setminus {\bf F}$ is an $(n-1)$-sphere.
The only non-trivial case is when ${\bf F}_0$ is non-empty.

Since ${\bf F}_0$ is the graph of a monotone function due to Theorem~\ref{th:matroid} (ii),
${\bf F}_0$ is a regular $(n-1)$-cell by the inductive hypothesis.
Thus, $\overline {\bf F}_0$, $\overline {\bf F}_+$, and $\overline {\bf F}_-$ are closed balls,
while $\overline {\bf F}_0 \setminus {\bf F}$ is an $(n-2)$-sphere.
Hence $\overline {\bf F}$ is obtained by gluing together two closed $n$-balls, $\overline {\bf F}_+$
and $\overline {\bf F}_-$ along closed $(n-1)$-ball $\overline {\bf F}_0$
(see Definition~\ref{def:gluing}).
Proposition~\ref{pr:cor3.16} implies that $\overline {\bf F}$ is a closed $n$-ball.

According to Proposition~\ref{pr:cor3.13}, the sets $\overline {\bf F}_+ \setminus {\bf F}
= \partial \overline {\bf F}_+ \setminus {\bf F}_0$ and
$\overline {\bf F}_- \setminus {\bf F} = \partial \overline {\bf F}_- \setminus {\bf F}_0$
are closed $(n-1)$-balls.
The frontier $\overline {\bf F} \setminus {\bf F}$ of ${\bf F}$ is obtained by gluing
$\overline {\bf F}_+ \setminus {\bf F}$ and $\overline {\bf F}_- \setminus {\bf F}$ along the set
$(\overline {\bf F}_+\cap \overline {\bf F}_-) \setminus {\bf F}$ which, by Lemma~\ref{cut},
is equal to $\overline {\bf F}_0 \setminus {\bf F}$ and thus,
is an $(n-2)$-sphere, the common boundary of $\overline {\bf F}_+ \setminus {\bf F}$
and $\overline {\bf F}_- \setminus {\bf F}$.
It follows from Proposition~\ref{pr:le1.10} that $\overline {\bf F} \setminus {\bf F}$
is an $(n-1)$-sphere.
\end{proof}

\begin{lemma}\label{complement}
If ${\bf F}$ and ${\bf F}_-$ in Lemma \ref{cut} are regular cells, then ${\bf F}_+$ is
also a regular cell.
\end{lemma}

\begin{proof}
Proposition~\ref{pr:shiota} implies that $\overline {\bf F}_+$ is
a closed $n$-ball.
By the global inductive hypothesis of the Theorem~\ref{th:regularcell}, ${\bf F}_0$ is a regular cell.
By Proposition~\ref{pr:cor3.13}, $\overline {\bf F}_+ \setminus {\bf F}=
\partial \overline {\bf F}_+ \setminus {\bf F}_0$ is a closed $(n-1)$-ball.
Then the frontier $\overline {\bf F}_+ \setminus {\bf F}_+$ of ${\bf F}_+$ is obtained by gluing
two closed $(n-1)$-balls, $\overline {\bf F}_+ \setminus {\bf F}$ and $\overline {\bf F}_0$ along the
$(n-2)$-sphere $\overline {\bf F}_0 \setminus {\bf F}$.
Therefore, by Proposition~\ref{pr:le1.10}, the frontier of ${\bf F}_+$ is an $(n-1)$-sphere.
\end{proof}

The following lemma is used on the inductive step of the proof of Theorem~\ref{th:regularcell}
and assumes the global inductive hypothesis (that a graph of a monotone map in less than
$n$ variables is a regular cell).

Introduce the notation $\Real_+:= \{ x \in \Real|\> x>0 \}$.

\begin{lemma}\label{le:curve_inside}
Let ${\bf F} \subset \Real_{+}^{n+k}$ be a graph of a monotone map $\f$ on a semi-monotone set
$X \subset {\rm span} \{ x_1, \ldots ,x_n \}$, such that the origin is in $\overline {\bf F}$.
Let $c(t)=(c_1(t),\dots,c_{n+k}(t))$
be a definable generic curve inside the smooth locus of ${\bf F}$ converging to the origin as $t \to 0$.
Then, for all small positive $t$, the set
$${\bf F}_t :={\bf F} \cap \{x_1<c_1(t),\dots,x_{n+k}<c_{n+k}(t)\}$$
is a cone with the vertex at the origin and a regular cell
as the base, i.e., ${\bf F}_t$ is a regular $n$-cell.
\end{lemma}

\begin{proof}
For a non-empty subset
$J:=\{ j_1, \ldots ,j_i\} \subset \{ 1, \ldots , n+k \}$, let
$$C_{J,t}:={\bf F} \cap\{ x_{j_1}=c_{j_1}(t),\ldots, x_{j_i}=c_{j_i}(t),\;
x_\ell<c_\ell(t)\;\text{for all}\; \ell \neq j_1,\dots,j_i\}.$$
Due to the theorem on triangulation of definable functions (\cite{Coste}, Th.~4.5),
for all small positive $t$, $\overline {\bf F}_t$ is definably homeomorphic to a closed cone
with the vertex at the origin and the base definably homeomorphic to $\overline C_t$, where $C_t$
is the union of non-empty $C_{J,t}$ for all non-empty $J$.
To complete the proof of the lemma, it is enough to show that $C_t$ is a regular cell.

According to Theorem~\ref{th:matroid} (ii), for every non-empty $J$
the $(n-i)$-dimensional set $C_{J,t}$ is either empty or a graph of a monotone map.
Hence, by the global inductive hypothesis, $C_{J,t}$ is a regular $(n-i)$-cell.
Its closure, $\overline C_{J,t}$, is a closed cell (i.e., is definably homeomorphic to
the closed cube $[0,1]^{n-i}$).
Note that if $J \subset K \subset \{ 1, \ldots , n+k \}$ then the closed  cell
$\overline C_{K,t}$ is a face of $\overline C_{J,t}$.

We prove by induction on $n$ the following claim.
If ${\bf F}$ is a graph in $\Real_{+}^{n+k}$ of a monotone map on
a semi-monotone set $X \subset \Real_{+}^{n}$, and $c(t)$ is a smooth point in ${\bf F}$
(i.e., we don't assume that that the origin is necessarily in $\overline {\bf F}$),
then $C_t$ is a regular cell.
The base for $n=1$ is obvious.
The case of an arbitrary $n$ we prove according to the following plan.
\begin{itemize}
\item[(a)]
Prove that for all $J$ the difference $\overline{C}_{J,t} \setminus \widehat C_{J,t}$,
where $\widehat C_{J,t}:= \bigcup_{K \supset J}C_{K,t}$, is a closed cell.
Then $\overline C_t$ is an $(n-1)$-dimensional cell complex (we will use the same notation
for a complex and its underlying polyhedron), consisting of the closed cells $\overline C_{J,t}$ and
the closed cells $\overline{C}_{J,t} \setminus \widehat C_{J,t}$ for all $J$.
\item[(b)]
Construct a linear cell complex, $\overline D_t$, similar to $\overline C_t$, replacing
${\bf F}$ by the tangent space to ${\bf F}$ at $c(t)$, and prove that $D_t$ is a regular cell.
\item[(c)]
Prove that $\overline C_t$ and $\overline D_t$ are abstractly isomorphic, which implies that
the pairs $(\overline C_t, C_t)$ and $(\overline D_t, D_t)$ are homeomorphic.
\end{itemize}

To prove (a) observe that, since $C_{J,t}$ is a regular $(n-i)$-cell, its boundary
$\partial C_{J,t}$ is the PL $(n-i-1)$-sphere, while by the inductive hypothesis the
difference $\widehat C_{J,t} \setminus C_{J,t}$ is a regular $(n-i-1)$-cell.
Since
$$(\overline C_{J,t} \setminus \widehat C_{J,t}) \cup
(\widehat C_{J,t} \setminus C_{J,t})= \partial C_{J,t},$$
the difference $\overline C_{J,t} \setminus \widehat C_{J,t}$ is a closed cell
by Newman's theorem (Corollary~3.13 in \cite{RS}).
Note that if $J \subset K \subset \{ 1, \ldots , n+k \}$ then the closed cell
$\overline C_{K,t} \setminus \widehat C_{K,t}$ is a face of
$\overline C_{J,t} \setminus \widehat C_{J,t}$.
It is clear that for any $J$ the closed cell
$\overline C_{J,t}$ is a cell complex of the required type.

Now we construct the cell complex to satisfy (b).
Recall that $c(t)$ is a smooth point of ${\bf F}$.
Let $L(t)$ be the tangent space to ${\bf F}$ at $c(t)$.
For every non-empty subset
$$J:=\{ j_1, \ldots ,j_i\} \subset \{ 1, \ldots , n+k \},$$
introduce the $(n-i)$-dimensional convex polyhedron
$$D_{J,t}:=L(t) \cap\{ x_{j_1}=c_{j_1}(t),\ldots, x_{j_i}=c_{j_i}(t),\;
x_\ell<c_\ell(t)\;\text{for all}\; \ell \neq j_1,\dots,j_i\},$$
and let $D_t$ be the union of sets $D_{J,t}$ for all non-empty $J$.
The same argument as in the case of $\overline C_{J,t}$, shows that the difference
$\overline D_{J,t} \setminus \widehat D_{J,t}$, where
$\widehat D_{J,t}:=\bigcup_{K \supset J}D_{K,t}$, is a closed cell.
Then $\overline D_{J,t}$ is a cell complex with closed cells of the kind $\overline D_{K,t}$,
for all $K \supset J$, and the unique closed cell $\overline D_{J,t} \setminus \widehat D_{J,t}$.
It follows that $\overline D_t$ is a cell complex with closed cells of the kind $\overline D_{J,t}$
and $\overline D_{J,t} \setminus \widehat D_{J,t}$ for all non-empty $J$.

Observe that $\overline D_t$ is a cone with the vertex $c(t)$ and the base
$B$ obtained by intersecting $D_t$ with a hyperplane in $\Real^{n+k}$ which separates
$c(t)$ from all other vertices of $\overline D_t$.
The base $B$ is the boundary of a convex polyhedron, and therefore a PL sphere.
It follows, using Lemma~1.10 in \cite{RS}, that $D_t$ is a regular cell.

To prove (c)
we claim that for each $J= \{j_1, \ldots ,j_i \}$, if the cell $C_{J,t}$ is non-empty then the cell
$D_{J,t}$ is non-empty, the converse implication being obvious.
Assume the opposite, i.e., that $C_{J,t} \neq \emptyset$ while $D_{J,t}= \emptyset$.
Then there exists
$\ell \in \{n+1, \ldots , n+k \} \setminus J$ such that the tangent space to $\overline C_{J,t}$ at
$c(t)$ lies in $\{ x_\ell = c_\ell \}$.
Since the map $\f$ is monotone, the component function $f_\ell$ of $\f$ is independent of
each variable $x_r$, where $r \in \{j_1, \ldots ,j_i \} \cap \{ 1, \ldots n \}$.
It follows that the graph $F_\ell$ of $f_\ell$ lies in $\{ x_\ell = c_\ell \}$
(cf. Remark~\ref{re:tengent}), therefore so does $C_{J,t}$.
This is a contradiction since, by the definition, $C_{J,t} \subset \{ x_\ell < c_\ell \}$.

Thus, we have a bijective correspondence between the regular cells in
$C_t$ and $D_t$.
Relating, in addition, for each $J$, the cell $\overline C_{J,t} \setminus \widehat C_{J,t}$
to the cell $\overline D_{J,t} \setminus \widehat D_{J,t}$ we obtain a bijective correspondence
between the closed cells in the cell complexes $\overline C_t$ and $\overline D_t$.
Note that the adjacency relations in both complexes are determined by the same simplicial subcomplex
of the simplex with vertices $\{ 1, \ldots ,n+k \}$, i.e., the complexes have the common nerve.
It follows that the complexes $\overline C_t$ and $\overline D_t$ are abstractly isomorphic (see
\cite{RS}), and their boundaries are abstractly isomorphic.
Lemma~2.18 in \cite{RS} implies that the sets $\overline C_t$ and $\overline D_t$ are homeomorphic,
and the boundaries $\partial \overline C_t$ and $\partial \overline D_t$ are homeomorphic.
Then, by Lemma~1.10 in \cite{RS}, the pairs $(\overline C_t, C_t)$ and
$(\overline D_t, D_t)$ are PL homeomorphic.
Therefore $C_t$ is a regular cell, since $D_t$ is a regular cell.
\end{proof}

We now generalize Lemma~\ref{le:curve_inside}, by removing the assumption that the curve $c(t)$
lies necessarily  inside ${\bf F}$.

\begin{lemma}\label{le:curve}
Let ${\bf F}$ be a graph in $\Real_{+}^{n+k}$ of a monotone map $\f$ on a semi-monotone set
$X \subset {\rm span} \{ x_1, \ldots ,x_n \}$, such that the origin is in $\overline {\bf F}$.
Let $c(t)=(c_1(t),\dots,c_{n+k}(t))$
be a definable curve inside
$\Real^{n+k}_+$ (not necessarily inside ${\bf F}$) converging to the origin as $t \to 0$.
Then, for all small positive $t$,
$${\bf F}_t :={\bf F} \cap \{x_1<c_1(t),\dots,x_n<c_{n+k}(t)\}$$
is a cone with the vertex at the origin and a regular cell $C_t$ as the base,
i.e., ${\bf F}_t$ is a regular $n$-cell.
\end{lemma}

\begin{proof}
We use the notations from the proof of Lemma~\ref{le:curve_inside}.
As in the proof of Lemma~\ref{le:curve_inside}, it is sufficient to show that $C_t$ is a regular cell.

Observe that for a small enough $t$, the set $\overline C_t$ is a link at the origin in
$\overline {\bf F}$.

Choose another definable curve, $s(t)$, converging the origin as $t \to 0$, so that $s(t)$
lies inside the smooth locus of ${\bf F}$ and is generic.
For a non-empty $J= \{ j_1, \ldots ,j_i \}$, let
$$S_{J,t}:={\bf F} \cap\{ x_{j_1}=s_{j_1}(t),\ldots, x_{j_i}=s_{j_i}(t),\;
x_\ell<s_\ell(t)\;\text{for all}\; \ell \neq j_1,\dots,j_i\},$$
and $S_t$ be the union of non-empty sets $S_{J,t}$ for all non-empty $J$.
According to Lemma~\ref{le:curve_inside}, $S_t$ is a regular cell.
For a small enough $t$, the closed cell $\overline S_t$ is also a link at the origin in
$\overline {\bf F}$.
By the theorem on the PL invariance of a link (\cite{RS}, Lemma~2.19), the two links
$\overline C_t$ and $\overline S_t$ are PL homeomorphic.

The same argument shows that the two links, $\partial C_t$ and $\partial S_t$, at the origin
in $\partial {\bf F}$ are PL homeomorphic.
Then, by Lemma~1.10 in \cite{RS}, the pairs $(\overline C_t, C_t)$ and
$(\overline S_t, S_t)$ are PL homeomorphic.
Therefore $C_t$ is a regular cell, since $S_t$ is a regular cell.
\end{proof}

\begin{lemma}\label{general_box}
Let ${\bf F}$ be a graph in $\Real_{+}^{n+k}$ of a monotone map $\f$ on a semi-monotone set
$X \subset {\rm span} \{ x_1, \ldots ,x_n \}$, such that the origin is in $\overline {\bf F}$,
and let $c=(c_1,\dots,c_{n+k}) \in \Real_{+}^{n+k}$.
Then ${\bf F}_c :={\bf F} \cap \{x_1<c_1,\dots,x_n<c_{n+k}\}$
is a regular cell for a generic $c$ with a sufficiently small $\| c \|$.
\end{lemma}

\begin{proof}
Consider a definable set
${\bf F}_{\y}:= {\bf F} \cap \{x_1<y_1,\dots,x_{n+k}<y_{n+k}\} \subset \Real_{+}^{2(n+k)}$
with coordinates $x_1, \ldots ,x_{n+k},y_1, \ldots ,y_{n+k}$ and $\y=(y_1, \ldots, y_{n+k})$.
By Corollary~\ref{cor:fibres}, there is a partition of $\Real_{+}^{n+k}$
(having coordinates $y_1, \ldots, y_{n+k}$) into definable sets $T$
such that if any $T$ is fixed, then
for all $\y \in T$ the closures $\overline {\bf F}_{\y}$ are
definably homeomorphic to the same polyhedron, and the frontiers
$\overline {\bf F}_{\y} \setminus {\bf F}_{\y}$ are definably homeomorphic to the same polyhedron.

For every $n$-dimensional $T$, such that the origin is in $\overline T$, there is, by the curve
selection lemma (\cite{Coste}, Th.~3.2) a definable curve $c(t)$ converging to $0$ as $t \to 0$.
Hence, by Lemma~\ref{le:curve}, for each $c \in T$ the
set $\overline {\bf F}_c$ is  a closed $n$-ball, while
$\overline {\bf F}_c \setminus {\bf F}_c$ is an $(n-1)$-sphere.
Therefore, ${\bf F}_c$ is a regular cell.
\end{proof}

\begin{lemma}\label{cutting1}
Using the notation from Lemma~\ref{general_box}, for
a generic $c \in \Real_{+}^{n+k}$ with a sufficiently small $\| c \|$, the intersection
$${\bf F}_c \cap \bigcap_{1 \le \nu \le \ell}  \{ x_{j_\nu} \sigma_\nu a_\nu \},$$
for any $\ell \le n+k$, $j_\nu \in \{ 1, \ldots , n+k \}$, $\sigma_\nu \in \{ <,> \}$, and for any
generic sequence $a_1 > \cdots > a_\ell$, is either empty or a regular cell.
\end{lemma}

\begin{proof}
It is sufficient to assume that $a_\nu < c_{j_\nu}$ for all $\nu$.
Induction on $\ell$.
For $\ell=1$, the set ${\bf F}_c \cap \{ x_{j_1} < a_1 \}$ is
itself a set of the kind ${\bf F}_c$, and therefore is a regular cell, by Lemma~\ref{general_box}.
Then the set ${\bf F}_c \cap \{ x_{j_1} > a_1 \}$ is a regular cell due to Lemma~\ref{complement}.

By the inductive hypothesis, every non-empty set of the kind
$$
{\bf F}_{c}^{(\ell-1)}:= {\bf F}_c \cap \bigcap_{1 \le \nu \le \ell-1}  \{ x_{j_\nu} \sigma_\nu a_\nu \}
$$
is a regular cell.
Also by the inductive hypothesis, replacing $c_{j_\ell}$ by $a_\ell$ if $a_\ell < c_{j_\ell}$,
every set
${\bf F}_{c}^{(\ell-1)} \cap \{ x_{j_\ell} < a_\ell \}$ is a regular cell.
Since both ${\bf F}_{c}^{(\ell-1)}$ and ${\bf F}_{c}^{(\ell-1)} \cap \{ x_{j_\ell} < a_\ell \}$
are regular cells, so is ${\bf F}_{c}^{(\ell-1)} \cap \{ x_{j_\ell} > a_\ell \}$,
by Lemma~\ref{complement}, which completes the induction.
\end{proof}

\begin{lemma}\label{box}
Let ${\bf F}$ be a graph in $\Real^{n+k}$ of a monotone map,
and let a point $\y=(y_1, \ldots ,y_{n+k})$ belong to ${\bf F}$.
Then for two generic points $a=(a_1, \ldots ,a_{n+k})$,
$b=(b_1, \ldots ,b_{n+k}) \in \Real_{+}^{n+k}$, with sufficiently small $\| a \|$ and $\| b \|$,
the intersection
$${\bf F}_{a,b}:= {\bf F} \cap \bigcap_{1 \le j \le n+k} \{ -a_j < x_j -y_j <b_j \}$$
is a regular cell.
\end{lemma}

\begin{proof}
Induction on $m:=n+k$ with the base $m=1$ ($n=0, k=1$) being obvious.

Translate the point $\y$ to the origin.
Let $c=(c_1, \ldots, c_m) \in \Real^m$ be a generic point, and ${\mathbb P}_{m,c}$ the open octant
of $\Real^m$ containing $c$.
By Lemma~\ref{general_box}, if $\| c \|$ is sufficiently small, the set
$${\bf F}_c:= {\bf F} \cap {\mathbb P}_{m,c} \cap \{|x_1| <|c_1|, \ldots , |x_m| < |c_m| \}$$
is either empty or a regular cell.
Choose such a point $c$ in every octant of $\Real^m$.

Choose $(-a_i)$ (respectively, $b_i$) as the maximum (respectively, minimum) among the negative
(respectively, positive) $c_i$ over all octants ${\mathbb P}_{m,c}$.
We now prove that, with so chosen $a$ and $b$, the set ${\bf F}_{a,b}$ is a regular cell.
Induction on a parameter $r=0, \ldots ,m-1$.
For the base of the induction, with $r=0$,
if $d=(d_1, \ldots ,d_m)$ is a vertex of
$$\bigcap_{1 \le j \le m} \{ -a_j < x_j <b_j \}$$
belonging to one of the $2^m=2^{m-r}$ octants of $\Real^m$, then ${\bf F}_d$
is either empty or a regular cell, by Lemma~\ref{cutting1}.
Partition the family of all sets of the kind ${\bf F}_d$ into pairs $({\bf F}_{d'},{\bf F}_{d''})$
so that $d'_1=a_1$, $d''_1=b_1$ and $d'_i=d''_i$ for all $i =2, \ldots, m$.
Whenever the cells ${\bf F}_{d'}$, ${\bf F}_{d''}$ are both non-empty, they have the common
$(n-1)$-face
$${\bf F} \cap (\{ 0 \} \times {\mathbb P}_{m-1,(d'_2, \ldots ,d'_m)})
\cap \{ x_1=0, |x_2| < |d'_2|, \ldots , |x_m|< |d'_m| \},$$
which, by the inductive hypothesis of the induction on $m$, is a regular cell.
Then, according to Corollary~\ref{union}, the union of the common face and
${\bf F}_{d'} \cup {\bf F}_{d''}$ is a regular cell.
Gluing in this way all pairs $({\bf F}_{d'},{\bf F}_{d''})$, we get a family of $2^{m-1}$
either empty or regular cells.
This family is partitioned into pairs of regular cells each of which has the common regular
cell face in the hyperplane $\{ x_2=0 \}$.
On the last step of the induction, for $r=m-1$, we are left with at most two regular cells
having, in the case of the exactly two cells, the common regular cell face in
the hyperplane $\{ x_m=0 \}$.
Gluing these sets along the common face, we get, by Corollary~\ref{union}, the regular cell
${\bf F}_{a,b}$.
\end{proof}

\begin{lemma}\label{box_cutting}
Using the notations from Lemma~\ref{box}, the intersection
\begin{equation}\label{eq:cutting}
V_{a,b}:= {\bf F}_{a,b} \cap \bigcap_{1 \le \nu \le \ell}  \{ x_{j_\nu} \sigma_\nu d_\nu \},
\end{equation}
for any $\ell \le n+k$, $j_\nu \in \{ 1, \ldots , n+k \}$, $\sigma_\nu \in \{ <,> \}$,
and for any generic $d_1 > \cdots > d_\ell$, is either empty or a regular cell.
\end{lemma}

\begin{proof}
Analogous to the proof of Lemmas~\ref{cutting1}.
\end{proof}

\begin{proof}[Proof of Theorem~\ref{th:regularcell}]
For each point $\y \in \overline {\bf F}$ choose generic points $a,\> b \in \Real^{n+k}$ as in
Lemma~\ref{box}, so that the set ${\bf F}_{a,b}$ becomes a regular cell.
We get an open covering of the compact set $\overline {\bf F}$ by the sets of the kind
$$A_{a,b}:= \bigcap_{1 \le j \le n+k} \{ -a_j < x_j -y_j <b_j \},$$
choose any finite subcovering ${\mathcal C}$.
For every $j=1, \ldots , n+k$ consider the finite set $D_j$ of $j$-coordinates $a_j,\> b_j$ for all
sets $A_{a,b}$ in ${\mathcal C}$.
Let
$$\bigcup_{1 \le j \le n+k} D_j= \{ d_1, \ldots ,d_L \}$$
with $d_1 > \cdots > d_L$.
Every set $V_{a,b}$, corresponding to a subset of a cardinality at most $\ell$ of
$\{ d_1, \ldots ,d_L \}$ (see (\ref{eq:cutting})), is a regular cell, by Lemma~\ref{box_cutting}.
The graph ${\bf F}$ is the union of those $V_{a,b}$ and their common faces, for which
$A_{a,b} \in {\mathcal C}$.

The rest of the proof is similar to the final part of the proof of Lemma~\ref{box}.
Use induction on $r=1, \ldots ,n+k$, within the current induction step of the
induction on $m=n+k$.
The base of the induction is for $r=1$.
Let $D_1= \{ d_{1,1}, \ldots , d_{1,k_1} \}$ with $d_{1,1} > \cdots > d_{1,k_1}$.
Partition the finite family of all regular cells $V_{a,b}$,
for all $A_{a,b} \in {\mathcal C}$, into $(|D_1|-1)$-tuples
so that the projections of cells in a tuple on the $x_1$-coordinate are exactly the intervals
\begin{equation}\label{eq:intervals}
(d_{1,k_1}, d_{1,k_1-1}), (d_{1,k_1-1}, d_{1,k_1-2}), \ldots , (d_{1,2}, d_{1,1}),
\end{equation}
and any two cells in a tuple having as projections two consecutive intervals in (\ref{eq:intervals})
have the common $(n-1)$-dimensional face in a hyperplane $\{ x_1= {\rm const} \}$.
This face, by the external inductive hypothesis (of the induction on $m$), is a regular cell.
According to Corollary~\ref{union}, the union of any two consecutive cells and their common face
is a regular cell.
Gluing in this way all consecutive pairs in every $(|D_1|-1)$-tuple, we get a smaller
family of regular cells.
This family, on the next induction step $r=2$, is partitioned into $(|D_2|-1)$-tuples of cells
such that in each of these tuples
two consecutive cells have the common regular cell face in a hyperplane $\{ x_2= {\rm const} \}$.
On the last step, $r=m$, of the induction we are left with one $(|D_n|-1)$-tuple of regular cells
such that two consecutive cells have the common regular cell face in a hyperplane
$\{ x_n= {\rm const} \}$.
Gluing all pairs of consecutive cells along their common faces, we get, by Corollary~\ref{union},
the regular cell ${\bf F}$.
\end{proof}

\subsection*{Graphs of monotone maps over real closed fields}

Fix an arbitrary real closed field $\rm R$.
In \cite{BGV} semi-algebraic semi-monotone sets in ${\rm R}^n$ were considered, and in particular
it was proved that every such set $X$ is a regular cell.
The latter means that there exists a {\em semi-algebraic} homeomorphism
$h:\> (\overline X, X) \to ([-1,1]^n, (-1,1)^n)$
(cf. Definition~\ref{def:regular}).

One can expand these results to semi-algebraic functions and maps over $\rm R$ (the graphs of
such functions and maps are semialgebraic sets).
In particular, the following statement is true.

\begin{theorem}
The graph ${\bf F} \subset {\rm R}^{n+k}$ of a semi-algebraic monotone map $\f: X \to {\rm R}^k$
on a semi-algebraic semi-monotone set $X \subset {\rm R}^n$ is a regular $n$-cell.
\end{theorem}

The proof of this theorem is based on applying the Tarski-Seidenberg transfer principle
(Proposition~5.2.3 in \cite{BCR}) to a first-order formalization of the statement of
Theorem~\ref{th:regularcell}, and is completely analogous to the proof of Theorem~3.3 in \cite{BGV}.

\section{Example: toric cubes}\label{sec:toric}

In \cite{Sturmfelsetal2012} Engstr\"om, Hersh and Sturmfels introduced a class of compact
semi-algebraic sets which they call \emph{toric cubes}.

The following definition is adapted from \cite{Sturmfelsetal2012}.
\begin{definition}\label{def:toric}
Let $\mathcal{A} = \{\mathbf{a}_1,\ldots,\mathbf{a}_n\} \subset \mathbb{N}^d$,
and $f_{\mathcal{A}}: [0,1]^d \rightarrow [0,1]^n$ be the map
\[
\mathbf{t} = (t_1,\ldots,t_d) \mapsto (\mathbf{t}^{\mathbf{a}_1},\ldots,
\mathbf{t}^{\mathbf{a}_n}),
\]
where $\mathbf{t}^{\mathbf{a}_i}:= t_{1}^{a_{i,1}} \cdots t_{d}^{a_{i,d}}$ for
$\mathbf{a}_i=(a_{i,1}, \ldots ,a_{i,d})$.
The image of $f_\mathcal{A}$ is called a toric cube.

We call the image of the restriction of $f_\mathcal{A}$ to $(0,1)^d$ an \emph{open toric cube}.
The closure of an open toric cube is a toric cube.
Note that an open toric cube is not necessarily an open subset of $\Real^n$, and need not be
contained in  $(0,1)^n$ (if some $\mathbf{a}_i = \mathbf{0}$).
\end{definition}

In this section we prove the following theorem.

\begin{theorem}
\label{thm:main}
An open toric cube $C \subset \Real^n$
is the graph of a monotone map.
\end{theorem}

As a result we obtain
\begin{corollary}\label{cor:toric}
\label{cor:main}
An open toric cube $C \subset [0,1]^n$, with $\dim (C) = k$,
is semi-algebraically homeomorphic to a standard open ball.
The pair $(\overline{C},C)$ is semi-algebraically homeomorphic to the pair
$([0,1]^{k},(0,1)^{k})$, in particular, a toric cube is semi-algebraically homeomorphic to
a standard closed ball.
\end{corollary}

\begin{remark}
Note that the first statement in Corollary~\ref{cor:main} is also proved in \cite{Sturmfelsetal2012},
Proposition~1.
In conjunction with Theorem~2 in \cite{Sturmfelsetal2012},
Corollary~\ref{cor:main} implies that any CW-complex in which the closures of each cell is a toric cube,
must be a regular cell complex, and this answers in the affirmative the
Conjecture~1 in \cite{Sturmfelsetal2012}.
\end{remark}

\begin{proof}[Proof of Theorem \ref{thm:main}]
Let $C \subset [0,1]^n$ be an open toric cube and suppose that $C=f_\mathcal{A}((0,1)^d)$ for
a monomial map $f_\mathcal{A}$ (see Definition~\ref{def:toric}).

Make the coordinate change $z_i = \log(t_i)$ for every $i=1, \ldots, d$, and take the logarithm
of every component of the map $f_\mathcal{A}$ expressed in coordinates $z_i$.
Denote the resulting map by $\log f_\mathcal{A}$.
Then $\log f_\mathcal{A}$ is the restriction of a linear map, namely
\[
\log f_\mathcal{A}: (-\infty,0)^d \rightarrow (-\infty,0)^n,
\]
defined by
\[
\mathbf{z} = (z_1,\ldots,z_d) \mapsto (\mathbf{a}_1\cdot\mathbf{z},
\ldots,\mathbf{a}_n\cdot\mathbf{z}).
\]

Observe that $\log$ (the component-wise logarithm) maps the open cube, $(0,1)^d$ (resp.
$(0,1)^n$) homeomorphically onto $(-\infty,0)^d$ (resp. $(-\infty,0)^n$).
It follows  that the fiber of the orthogonal projection of $C$ to any $k$-dimensional
coordinate subspace is the pre-image under the $\log$ map
of an affine subset of $(-\infty,0)^n$, and is a single point if it is zero-dimensional.
Hence $C$ is a graph of a quasi-affine map (choose any set of $k$ coordinates such that
the image of $C$ under the orthogonal projection to the coordinate subspace of those
coordinates is full dimensional).

Similarly, the intersection of $C$ with any affine coordinate subspace is the
pre-image under the $\log$ map, of an affine subset of $(-\infty,0)^n$ and hence connected.

We proved that $C$ satisfies the conditions of Theorem~\ref{th:def_monotone_map}, hence
$C$ is the graph of a monotone map.
\end{proof}

\begin{proof}[Proof of Corollary \ref{cor:main}]
Immediate consequence of Theorem \ref{thm:main} and
Theorem \ref{th:regularcell}.
\end{proof}

\section{Appendix}\label{sec:appendix}

Here we formulate some propositions, mostly from PL topology, which are used in the proofs above.

\begin{proposition}[\cite{Coste}, Theorem~4.4]\label{pr:triang}
Let $X \subset \Real^n$ be a compact definable set and let $Y_i,\> i=1, \ldots ,k$ be
definable subsets of $X$.
Then there exists a finite simplicial complex $K$ and a definable homeomorphism (triangulation)
$\varphi:\> K \to X$ such that each $Y_i$ is a union of images by $\varphi$ of open simplices of $K$.
\end{proposition}

Proposition~\ref{pr:triang} implies, in particular, every compact definable set in $\Real^n$ is
a {\em polyhedron} \cite{RS}.

Let $\sim$ (respectively, $\sim_{PL}$) denote the relation of definable (respectively, PL) homeomorphism.

\begin{lemma}\label{le:pl_hom}
Let $X,\> Y \subset \Real^n$ be two definable compact sets, and $\widetilde X,\> \widetilde Y$
two polyhedra, such that $X \sim \widetilde X$, $Y \sim \widetilde Y$, and
$\widetilde X \sim_{PL} \widetilde Y$.
Then $X \sim Y$.
\end{lemma}

\begin{proof}
Straightforward, since, by Theorems~2.11, 2.14 in \cite{RS}, any PL homeomorphism of compact
polyhedra is definable.
\end{proof}

\begin{definition}\label{def:regular}
A definable set $X$ is called a (topologically) regular $m$-cell if the pair $(\overline X,X)$ is
definably homeomorphic to the pair $([-1, 1]^m, (-1,1)^m)$.
\end{definition}

\begin{definition}\label{def:gluing}
Let $Z$ be a closed (open) PL $(n-1)$-ball, $X$, $Y$ be closed (respectively, open) PL $n$-balls,
and $$\overline Z= \overline X \cap \overline Y= \partial X \cap \partial Y.$$
We say that $X \cup Y \cup Z$ is obtained by {\em gluing} $X$ {\em and} $Y$ {\em along} $Z$.
\end{definition}

\begin{proposition}[\cite{RS}, Lemma~1.10]\label{pr:le1.10}
Let $X$ and $Y$ be closed PL $n$-balls and $h:\> \partial X \to \partial Y$ a PL homeomorphism.
Then $h$ extends to a PL homeomorphism $h_1:\> X \to Y$.
\end{proposition}

\begin{proposition}[\cite{RS}, Corollary~$3.13_n$]\label{pr:cor3.13}
Let $X$ be a closed PL $n$-ball, $Y$ be a closed $(n+1)$-ball, $\partial Y$ be its boundary
(the PL $n$-sphere), and let $X \subset \partial Y$.
Then $\overline {\partial Y \setminus X}$ is a PL $n$-ball.
\end{proposition}

\begin{proposition}[\cite{RS}, Corollary~3.16]\label{pr:cor3.16}
Let $X$, $Y$, $Z$ be  closed PL balls, as in Definition~\ref{def:gluing},
and $X \cup Y$ be obtained by gluing $X$ and $Y$ along $Z$.
Then $X \cup Y$ is a closed PL $n$-ball.
\end{proposition}

\begin{proposition}[\cite{Shiota}, Lemma~I.3.8]\label{pr:shiota}
Let $X, Y \subset \Real^n$ be compact polyhedra such that $X$ and $X \cup Y$ are
closed PL $n$-balls.
Let $X \cap Y$ be a closed PL $(n-1)$-ball contained in $\partial X$, and let the interior
of $X \cap Y$ be contained in the interior of $X \cup Y$.
Then $Y$ is a closed PL $n$-ball.
\end{proposition}

\begin{proposition}[\cite{VDD}, Ch.~8, (2.14)]\label{pr:fibres}
Let $X \subset \Real^{m+n}$ be a definable set, and let $\pi:\> \Real^{m+n} \to \Real^m$ be the
projection map.
Then there exist an integer $N>0$ and a definable (not necessarily continuous) map
$f:\> X \to \Delta$, where $\Delta$ is an $(N-1)$-simplex, such that for every $\x \in \Real^m$
the restriction $f_{\x}:\> (X \cap \pi^{-1}(\x))  \to \Delta$ of $f$ to $X \cap \pi^{-1}(\x)$
is a definable homeomorphism onto a union of faces of $\Delta$.
\end{proposition}

\begin{corollary}\label{cor:fibres}
Using the notations from Proposition~\ref{pr:fibres},
let all fibres $X \cap \pi^{-1} (\x)$ be definable compact sets.
Then there is a partition
of $\pi (X)$ into a finite number of definable sets $T \subset \Real^m$ such that all
fibres $X \cap \pi^{-1} (\x)$ with $\x \in T$ are definably homeomorphic, moreover
each of these fibres is definably homeomorphic to the same simplicial complex.
\end{corollary}

\begin{proof}
There is a finite number of different unions of faces in $\Delta$.
Since $f$ is definable, the pre-image of any such union under the map $f \circ \pi^{-1}$
is a definable set.
\end{proof}

\end{document}